\documentclass[]{amsart}
\usepackage{amssymb,amsrefs, amsmath,mathrsfs,enumitem,upgreek,mathtools}
\usepackage{microtype}

\DefineSimpleKey{bib}{primaryclass}{}
\DefineSimpleKey{bib}{archiveprefix}{}

\BibSpec{arXiv}{%
 +{}{\PrintAuthors}{author}
 +{,}{ \textit}{title}
 +{}{ \parenthesize}{date}
 +{,}{ arXiv }{eprint}
 +{,}{ primary class }{primaryclass}
}

\newcommand{\AC}{\mathsf{AC}}			% Axiom of choice
\newcommand{\AComega}{{\mathsf{AC}}_\omega}	% Axiom of countable choice

\newcommand{\DC}{\mathsf{DC}}			% Dependent choice
\newcommand{\forcing}[1]{\mathbf{ #1}} 	% font for forcing notions
\DeclareMathOperator{\Cl}{Cl} 		% closure
\newcommand{\Nbhd}{{\boldsymbol N}\!} % basic nbhd
\newcommand{\FORALL}[1]{\forall {#1} \, }
\newcommand{\EXISTS}[1]{\exists {#1} \, }
\newcommand{\Q}{\mathbb{Q}}	% the rationals
\newcommand{\R}{\mathbb{R}}	% the reals
\newcommand{\ZF}{\mathsf{ZF}} 		% Zermelo-Fraenkel
\newcommand{\Pow}{\mathscr{P}}		% power set
\newcommand{\IMPLIES}{\Rightarrow }
\newcommand{\IFF}{\Leftrightarrow}
\newcommand{\AND}{\mathbin{\, \wedge \,}}
\newcommand{\markdef}[1]{\textbf{#1}} %	markdef
\newcommand{\Mid}{\boldsymbol\mid}			% separator
\newcommand{\setof}[2]{\mathopen \{{#1}\Mid{#2} \mathclose\}} % set of...
\newcommand{\setofLR}[2]{\left \{{#1} \Mid {#2} \right\}} % set of... with large delimiters
\newcommand{\set}[1]{\mathopen \{ {#1} \mathclose \}} % set 
\newcommand{\setLR}[1]{\left \{ {#1} \right \}} % set with large delimiters
\newcommand{\seqof}[2]{\mathopen \langle #1 \Mid #2 \mathclose \rangle} % sequence
\newcommand{\seq}[1]{\mathopen \langle #1 \mathclose \rangle}	% sequence of objects
\newcommand{\seqLR}[1]{\left \langle #1 \right \rangle}	% sequence of objects large delimiters
\DeclareMathOperator{\lh}{lh}				% length
\newcommand{\conc}{{}^\smallfrown}		% concatenation symbol
\DeclareMathOperator{\dom}{dom} 		% domain
\DeclareMathOperator{\ran}{ran} 		% range
\newcommand{\Vv}{\mathord{\mathrm{V}}}		% V
\newcommand{\Ll}{\mathord{\mathrm{L}}}		% L
\newcommand{\Gdelta}{\mathbf{G}_\delta}	% G_delta
\newcommand{\Fsigma}{\mathbf{F}_\sigma}	% F_sigma
\newcommand{\Gdeltadelta}{\mathbf{G}_{\delta\delta}}	% G_deltadelta
\newcommand{\Fsigmasigma}{\mathbf{F}_{\sigma\sigma}}	% F_sigmasigma
\newcommand{\bSigma}{\boldsymbol{\Sigma}}
\newcommand{\bPi}{\boldsymbol{\Pi}}
\newcommand{\bDelta}{\boldsymbol{\Delta}}
\DeclareMathOperator{\Aut}{Aut} % automorphism group
\newcommand{\forces}{\Vdash}
\newcommand{\HS}{\mathsf{HS}} %Hereditarily symmetric
\newcommand{\FL}{\mathsf{FL}} % Feferman-Levy
\DeclareMathOperator{\Fix}{Fix}
\newcommand{\Models}{\vDash}

\renewcommand{\restriction}{\mathop{\upharpoonright}}
\newenvironment{enumerate-(a)}{\begin{enumerate}[label={\upshape (\alph*)}]}{\end{enumerate}}
\theoremstyle{plain}
\newtheorem{theorem}{Theorem}[section]

\newtheorem{question}[theorem]{Question}

\newtheorem{proposition}[theorem]{Proposition}
\newtheorem{lemma}[theorem]{Lemma}
\newtheorem{corollary}[theorem]{Corollary}
\newtheorem{claim}{Claim}[theorem]
\newtheorem{subclaim}{Subclaim}[claim]
\theoremstyle{definition}
\newtheorem{definition}[theorem]{Definition}
\theoremstyle{remark}
\newtheorem{remark}[theorem]{Remark}

\newtheorem*{remarks*}{Remarks}

\usepackage{hyperref}

\begin{document}

\title{Does \( \DC \) imply \( \AComega \), uniformly?}
\author{Alessandro Andretta}
\address{Università degli Studi di Torino, Dipartimento di Matematica ``G. Peano", Via Carlo Alberto 10, 10123 Torino, Italy}
\curraddr{}
\email{alessandro.andretta@unito.it}
\author{Lorenzo Notaro}
\address{Università degli Studi di Torino, Dipartimento di Matematica ``G. Peano", Via Carlo Alberto 10, 10123 Torino, Italy}
\curraddr{}
\email{lorenzo.notaro@unito.it}

\thanks{The authors are grateful to the anonymous referee for useful comments and suggestions. 
This research was supported by the project PRIN 2017 ``Mathematical Logic: models, sets, computability", prot. 2017NWTM8R. 
The second author would also like to acknowledge INdAM for the financial support.}


\begin{abstract}

The Axiom of Dependent Choice \( \DC \) and the Axiom of Countable Choice \( \AComega \) are two weak forms of the Axiom of Choice that can be stated for a specific set: \( \DC ( X ) \) asserts that any total binary relation on \( X \) has an infinite chain, while \( \AComega ( X ) \) asserts that any countable collection of nonempty subsets of \( X \) has a choice function.
It is well-known that \( \DC \IMPLIES \AComega \). 
We study for which sets and under which hypotheses \( \DC ( X ) \IMPLIES \AComega ( X ) \), and then we show it is consistent with \( \ZF \) that there is a set \( A \subseteq \R \) for which \( \DC ( A ) \) holds, but \( \AComega ( A ) \) fails.
\end{abstract}
%\date{}
\maketitle

\section{Introduction}
The \markdef{Axiom of Choice} \( \AC \) is the statement \( \FORALL{X} \AC ( X ) \), where 
\begin{equation}
 X \neq \emptyset \IMPLIES \EXISTS{f \colon \Pow ( X ) \to X}\, \FORALL{ A \subseteq X} ( A \neq \emptyset \IMPLIES F ( A ) \in A ) . \tag{\( \AC ( X ) \)}
\end{equation}
The function \( f \) is a choice function for \( X \).
Observe that \( \AC ( X ) \) if and only if ``\( X \) can be well-ordered''.

By restricting the choice function we have that \( \AC ( X ) \IMPLIES \AC_I ( X ) \), where 
\begin{equation*}%\label{}
\parbox[c]{0.8\textwidth}{For any sequence \( ( A_i )_{ i \in I} \) of nonempty subsets of \( X \) there is \( ( a_i )_{ i \in I} \) such that \( \FORALL{ i \in I } ( a_i \in A_i ) \).} \tag{\( \AC_I ( X ) \)}
\end{equation*}
Of particular interest is the case when \( I = \omega \): the \markdef{Axiom of Countable Choice} \( \AComega \) is \( \FORALL{X} \AComega ( X ) \).
(In the literature \( \mathsf{CC} \) is another name for this axiom.)

Let \( R \) be a binary relation on a set \( X \).
\begin{itemize}
\item
An \( R \)\markdef{-chain} is a sequence \( ( x_n )_{ n \in \omega } \) of elements of \( X \) such that \( x_i \mathrel{R} x_{ i + 1} \) for all \( i \in \omega \).
The element \( x_0 \) is the starting point of the chain.
\item
An \( R \)\markdef{-cycle} is a finite string \( x_0 , \dots , x_n \) of elements of \( X \) such that \( x_i \mathrel{R} x_{ i + 1} \) for all \( i < n \) and \( x_n \mathrel{R} x_0 \).
\item
\( R \) is \markdef{total on \( X \)} if \( \FORALL{x \in X } \EXISTS{y \in X} x \mathrel{R} y \).
\end{itemize}
Any \( R \)-cycle yields an \( R \)-chain.

The \markdef{Axiom of Dependent Choice} \( \DC \) is \( \FORALL{X} \DC ( X ) \), where
\begin{equation*}
\parbox[c]{0.8\textwidth}{For any nonempty, total \( R \subseteq X^2 \) there is \( ( x_n )_{n \in \omega} \) such that \( \forall n \in \omega \left ( x_n \mathrel{R} x_{n + 1} \right ) \).} \tag{\( \DC ( X ) \)}
\end{equation*}

The axioms \( \DC \) and \( \AComega \) are ubiquitous in set theory and figure prominently in many areas of mathematics, including analysis and topology.
They are probably the most popular weak forms of the axiom of choice, since they are powerful enough to enable standard mathematical constructions, yet they are weak enough to avoid the pathologies given by \( \AC \).

It is well-known that \( \DC \IMPLIES \AComega \) (Theorem~\ref{th:folklore}), so one may ask if this result holds uniformly, that is: does \( \DC ( X ) \IMPLIES \AComega ( X ) \) for all \( X \)?
This implication holds for many \( X \)s, but in order to prove it in general, \( \AComega ( \R ) \) must be assumed (Theorem~\ref{th:Alessandro}).
In Section~\ref{sec:main}, we will show that the assumption \( \AComega ( \R ) \) cannot be dropped, as it is consistent with \( \ZF \) that there is a set \( A \subseteq \R \) for which \( \DC ( A ) \) holds, but \( \AComega ( A ) \) fails (Theorem~\ref{th:Lorenzo}). 
In Section~\ref{sec:complementary}, we discuss some complementary results along with the question on the definability of the set constructed in Section~\ref{sec:main}.

\subsection*{Notation}
Our notation is standard---see e.g.~\cite{Jech:2003pd}.
We write \( X \precsim Y \) to say that there is an injection from \( X \) into \( Y \), and \( X \approx Y \) to say that \( X \) and \( Y \) are in bijection.
Ordered pairs are denoted by \( ( a , b ) \), finite sequences are denoted by \( \seq{ a_0 , \dots , a_n } \) or by \( ( a_0 , \dots , a_n ) \), countable sequences are denoted by \( \seqof{ a_n }{ n \in \omega } \) or by \( ( a_n )_{ n \in \omega } \).
The \markdef{concatenation} of a finite sequence \( s \) with a finite/countable sequence \( t \) is the finite/countable sequence \( s \conc t \) obtained by listing all elements of \( s \) and then all elements of \( t \).
The set of all finite (countable) sequences from \( X \) is \( {}^{< \omega } X \) (respectively: \( {}^ \omega X \)).
The collection of all finite subsets of a set \( X \) is \( [ X ]^{< \omega } \).

If \( Y \) is a subset of a topological space \( X \), then \( \Cl ( Y ) \) is its closure, and \( \Cl_A ( Y ) \coloneqq \Cl ( Y ) \cap A \) is the closure of \( Y \cap A \) with respect to \( A \subseteq X \).

Following set-theoretic practice, we refer to members of \( {}^\omega \omega \) or \( \Pow ( \omega ) \) as ``reals", and we identify \( \R \) with the Baire space \({}^{\omega} \omega\) or with the Cantor space \( {}^ \omega 2 \), depending on what is most convenient for the argument at hand.

\section{Basic constructions}
For the reader's convenience, let us recall a few notions and results that will be used throughout the paper.

A set \( X \) is \markdef{finite} if \( X \approx n \) for some \( n \in \omega \); otherwise it is \markdef{infinite}.
A set \( X \) is \markdef{Dedekind-infinite} if \( \omega \precsim X \); otherwise it is \markdef{Dedekind-finite} or simply \markdef{D-finite}.
Every finite set is D-finite, and assuming \( \AComega \) the converse holds.

It is consistent with \( \ZF \) that infinite D-finite sets exist (see Section~\ref{subsec:firstCohen}).
By~\cite{KaragilaarXiv}, it is even consistent that every set is the surjective image of a D-finite set.

Let \( R \) be a binary relation.
With abuse of notation, we write
\[
R ( x ) \coloneqq \setof{ y }{ x \mathrel{R} y }
\]
for the set of all \( y \)s that are related to \( x \), and 
\[
 R \restriction A \coloneqq R \cap ( A \times A ) 
\]
for the restriction of \( R \) to the set \( A \).
The \markdef{transitive closure} of \( R \) 
\[
R^+ \coloneqq \setof{ ( x , y ) }{ \EXISTS{\seq{ y_0 , \dots , y_n } } ( x \mathrel{R} y_0 \mathrel{R} y_1 \mathrel{R} \cdots \mathrel{R} y_n \mathrel{R} y ) }
\]
is the smallest transitive relation containing \( R \).

The next few results are folklore.

\begin{proposition}\label{prop:basicpropertiesofDC&AComega}
Let \( X \) be a set.
\begin{enumerate-(a)}
\item \label{prop:basicpropertiesofDC&AComega-a} 
If \( Y \) is the surjective image of \( X \), then \( \DC ( X ) \IMPLIES \DC ( Y ) \).
\item \label{prop:basicpropertiesofDC&AComega-b} 
\( \DC ( X ) \) is equivalent to the seemingly stronger statement:
For any total \( R \subseteq X \times X \) and for any \( a \in X \), there is an \( R \)-chain starting from \( a \).
\item \label{prop:basicpropertiesofDC&AComega-d} 
If \( \emptyset \neq A_n \subseteq X \) and \( A_n \cap A_m = \emptyset \), then \( \DC ( X ) \) implies that there is a choice function for the \( A_n \)'s.
\item \label{prop:basicpropertiesofDC&AComega-e} 
\( \DC ( X \times \omega ) \IMPLIES \AComega ( X ) \).
\end{enumerate-(a)}
\end{proposition}

\begin{proof}
\ref{prop:basicpropertiesofDC&AComega-a}
Assume \( \DC ( X ) \) and let \( R \) be a total relation on \( Y \) and let \( F \colon X \to Y \) be a surjection.
The relation \( S = \setof{ ( x , x' ) \in X^2 }{ ( F ( x ) , F ( x ' ) ) \in R } \) is total on \( X \), so by assumption there is an \( S \)-chain \( ( x_n )_{ n \in \omega } \).
Then \( ( F ( x_n ) )_{ n \in \omega } \) is an \( R \)-chain.

\smallskip

\ref{prop:basicpropertiesofDC&AComega-b}
Suppose \( R \subseteq X^2 \) is total and let \( a \in X \).
Observe that \( S = R \restriction R^+ ( a ) \) is total on \( R^+ ( a ) \).
By part~\ref{prop:basicpropertiesofDC&AComega-a} \( \DC ( R^+ ( a ) ) \) holds, hence there is an \( S \)-chain \( ( y_n )_{ n \in \omega } \).
Let \( ( x_0 , \dots , x_{ k + 1} ) \) witness that \( y_0 \in R^+ ( a ) \), i.e.~\(x_0 = a\), \( x_{k+1} = y_0 \) and \(x_i \mathrel{R} x_{i+1}\) for all \(i \le k\): then \( ( x_0 , \dots , x_k ) \conc ( y_n )_{ n \in \omega } \) is an \( R \)-chain starting from \( a \). 

\smallskip

\ref{prop:basicpropertiesofDC&AComega-d}
Let \( R \) be the relation on \( \bigcup_n A_n \subseteq X \) defined by 
\[
x \mathrel{R} y \IFF \EXISTS{ n \in \omega } \left ( x \in A_n \wedge y \in A_{n + 1} \right )
\]
By part~\ref{prop:basicpropertiesofDC&AComega-a} \( \DC ( \bigcup_n A_n ) \) holds, hence by part~\ref{prop:basicpropertiesofDC&AComega-b} there is an \( R \)-chain \( ( a_n )_{ n \in \omega } \) in \( \bigcup_n A_n \) starting from any \( a_0 \in A_0 \).
Observe that any \( R \)-chain \( ( a_n )_{n \in \omega } \) with \(a_0 \in A_0\) is such that \( a_n \in A_n \) for all \( n \in \omega \). 

\smallskip

\ref{prop:basicpropertiesofDC&AComega-e}
Given \( \emptyset \neq A_n \subseteq X \), let \( \bar{A}_n = A_n \times \setLR{n} \subseteq X \times \omega \).
By hypothesis and part \ref{prop:basicpropertiesofDC&AComega-d}, there is a sequence \( ( a_n , n)_{n \in \omega} \) such that \( (a_n , n ) \in \bar{A}_n \), hence \( a_n \in A_n \).
\end{proof}

The gist of part~\ref{prop:basicpropertiesofDC&AComega-d} of Proposition~\ref{prop:basicpropertiesofDC&AComega} is that we can use dependent choice rather than countable choice whenever the sets we choose from are disjoint.
Here is an example of such an application.

\begin{lemma}\label{lem:firstcountable}
Suppose \( X \) is a first countable \( \mathrm{T}_1\) space and \( a \in \Cl ( A ) \setminus A \) where \( A \subseteq X \).
Assume \( \DC ( A ) \) holds.
Then there are distinct \( a_n \in A \) such that \( a_n \to a \).
In particular \( \omega \precsim A \).
\end{lemma}

\begin{proof}
Let \( \setof{ U_n }{ n \in \omega } \) be a neighborhood base for \( a \), with \(U_{n+1}\subseteq U_n\) for every \(n\). 
Given that \( X \) is \( \mathrm{T}_1 \), we can assume, by passing to a subsequence if needed, that \( A_n = ( U_n \setminus U_{ n + 1 } ) \cap A\) is nonempty for every \( n \). 
Since the \( A_n \)s are pairwise disjoint and nonempty, by Proposition~\ref{prop:basicpropertiesofDC&AComega}\ref{prop:basicpropertiesofDC&AComega-d} there is a sequence of \( ( a_n ) _{ n \in \omega } \) of distinct elements of \( A \) such that \( a_n \in A_n \) for every \( n \). 
\end{proof}

\begin{lemma}\label{lem:sums}
Let \( X \) be a set.
\begin{enumerate-(a)}
\item\label{lem:sums-a}
\( X \times 2 \precsim X \IMPLIES X \times \omega \precsim X \).
\item\label{lem:sums-b}
If \( X \neq \emptyset \), then \( {}^{< \omega } ( {}^{< \omega } X ) \precsim {}^{< \omega } X \), so \( {}^{< \omega } X \times 2 \precsim {}^{< \omega } X \).
\item\label{lem:sums-c}
\( \FORALL{X} \EXISTS{Y} ( X \subseteq Y \wedge {}^{< \omega } Y \precsim Y ) \).
\end{enumerate-(a)}
\end{lemma}

\begin{proof}
\ref{lem:sums-a}
If \( f_0 , f_1 \colon X \to X \) are injections with \( \ran ( f_0 ) \cap \ran ( f_1 ) = \emptyset \), then define an injection \( F \colon X \times \omega \to X \) as follows:
\[
F ( x , 0 ) = f_0 ( x ), \qquad F ( x , n + 1 ) = \underbrace{f_1\circ \dots \circ f_1}_{n+1 \text{ times}} \circ f_0 ( x ) . 
\]

\ref{lem:sums-b}
If \( X \) is a singleton, then \( {}^{< \omega } X \approx \omega \), and the result follows at once.
If \( X \) has at least two elements, the result follows from~\cite{Andretta:2022aa}*{Proposition 2.1}.

\ref{lem:sums-c}
Given \( X \) take \( Y = \Vv_ \lambda \) with sufficiently large limit \( \lambda \).
\end{proof}
 
From Lemma~\ref{lem:sums} and Proposition~\ref{prop:basicpropertiesofDC&AComega}\ref{prop:basicpropertiesofDC&AComega-e} we obtain at once:
 
\begin{theorem}\label{th:folklore}
\begin{enumerate-(a)}
\item\label{th:folklore-a}
If \( X \times 2 \precsim X \), then \( \DC ( X ) \IMPLIES \AComega ( X ) \).
In particular: \( \DC ( \R ) \IMPLIES \AComega ( \R ) \).
\item\label{th:folklore-b}
\( \FORALL{X} \EXISTS{Y} ( X \subseteq Y \wedge ( \DC ( Y ) \IMPLIES \AComega ( Y ) ) \). 
\item\label{th:folklore-c}
\( \DC \IMPLIES \AComega \). 
\end{enumerate-(a)}
\end{theorem}

\begin{lemma}\label{lem:propertiesofA}
\begin{enumerate-(a)}
\item\label{lem:propertiesofA-a}
Let \( A \subseteq \R \).
Then \( \AComega ( A ) \IMPLIES A \) is separable.
\item\label{lem:propertiesofA-b}
\( \AComega ( \R ) \IFF \FORALL{A \subseteq \R} ( A \text{ is separable} ) \).
\item\label{lem:propertiesofA-c}
Suppose \( A \subseteq \R \) contains a nonempty perfect set, and assume \( \DC ( A ) \).
Then \( \DC ( \R ) \) holds, and hence \( \AComega ( A ) \) holds.
\end{enumerate-(a)}
\end{lemma}

\begin{proof}
As \( A \) is second countable, part~\ref{lem:propertiesofA-a} of Lemma~\ref{lem:propertiesofA} follows.

\ref{lem:propertiesofA-b}
The direction \( ( \IMPLIES ) \) is a direct consequence of part~\ref{lem:propertiesofA-a}. 
For the other direction, fix a sequence \( ( A_n )_{ n \in \omega } \) of nonempty subsets of \( {}^ \omega \omega \) and consider the set \( A = \setof{ \seq{ n } \conc x }{ n \in \omega \text{ and } x \in A_n } \). 
From an enumeration of a dense subset of \( A \) (which exists by assumption), we can extract a choice function for \( ( A_n )_{ n \in \omega } \).

\ref{lem:propertiesofA-c}
If \( P \subseteq A \) is perfect, then \( P \approx \R \), and since \( A \) surjects onto \( P \), then \( \DC ( \R ) \) holds, and hence \( \AComega ( \R ) \) holds.
%Then \( \AComega ( A ) \) follows from \( A \subseteq \R \).
\end{proof}

Note that the implication in part~\ref{lem:propertiesofA-a} of Lemma~\ref{lem:propertiesofA} cannot be reversed: if \( A \subseteq \R \) is a witness of the failure of countable choice, then the same is true of the separable set \( A \cup \Q \).

\subsection{\( \AComega ( X ) \) follows from \( \DC ( X ) \) together with \( \AComega ( \R ) \).}
Let us start with the following combinatorial result that might be of independent interest.
It is stated for families of sets indexed by an arbitrary set \( I \), but when \( I = \omega \) the assumption \( \AC_I ( \Pow ( I ) ) \) becomes \( \AComega ( \R ) \).

\begin{lemma}\label{lem:combinatorial}
Let \( ( X_i )_{ i \in I } \) be nonempty sets, and assume \( \AC_I ( \Pow ( I ) ) \).
Then there are \( ( Y_i )_{ i \in I } \) such that \( \emptyset \neq Y_i \subseteq X_i \) and for all \( i , j \in I \) either \( Y_i = Y_j \) or else \( Y_i \cap Y_j = \emptyset \).
\end{lemma}

\begin{proof}
Let \( F \colon \bigcup_{i \in I} X_i \to \Pow ( I ) \), \( F ( x ) = \setof{ i \in I }{ x \in X_i } \) and let \( A_i = \setof{ a \in \ran ( F ) }{ i \in a } \). 
Observe that for all \( x \in \bigcup_{i \in I} X_i \) and all \( i \in I \)
\begin{equation}\label{eq:lem:combinatorial}
x \in X_i \IFF F ( x ) \in A_i . 
\end{equation}
In particular, \( \emptyset \neq A_i \subseteq \Pow ( I ) \) for all \( i \in I \). 
By \( \AC_I ( \Pow ( I ) ) \) pick \( a_i \in A_i \), and let \( Y_i = F^{-1} ( \setLR{a_i} ) \subseteq \bigcup_{i \in I} X_i \).
Then 
\[ 
Y_i = \setofLR{ x }{ F ( x ) = a_i } = \setof{ x }{ \setof{ j }{ x \in X_j} = a_i },
\]
and since \( i \in a_i \), then \( Y_i \subseteq X_i \). 
The sets \( Y_i \) need not be distinct as the \( a_i \)s need not be distinct, but if \( a_i \neq a_ j \), then \( Y_i \cap Y_j = \emptyset \).
\end{proof}

By~\eqref{eq:lem:combinatorial} if the \( X_i \)s are finite, then so are the \( A_i \)s. 
If \( \Pow ( I ) \) is linearly orderable (e.g.~when \( I \) is well-orderable), then the \( a_i \)s can be chosen without appealing to any axiom.
Therefore:

\begin{corollary}\label{cor:anteAlessandro}
If \( \Pow ( I ) \) is linearly orderable and \( ( X_i )_{ i \in I } \) are finite, nonempty sets, then there are \( \emptyset \neq Y_i \subseteq X_i \) such that for all \( i , j \in I \) either \( Y_i = Y_j \) or else \( Y_i \cap Y_j = \emptyset \).
\end{corollary}

\begin{theorem}\label{th:Alessandro}
Assume \( \AComega ( \R ) \), then \( \FORALL{X} ( \DC ( X ) \IMPLIES \AComega ( X ) ) \).
\end{theorem}

\begin{proof}
Assume \( \DC ( X ) \) and let \( \emptyset \neq X_n \subseteq X \) for \( n \in \omega \).
By Lemma~\ref{lem:combinatorial}, there are \( \emptyset \neq Y_n \subseteq X_n \) such that for all \( n , m \in \omega \) either \( Y_n = Y_m \) or else \( Y_n \cap Y_m = \emptyset \).
Let \( I \subseteq \omega \) be such that \( \setof{ Y_i }{ i \in I } = \setof{ Y_n }{ n \in \omega } \) and \(Y_i \cap Y_j = \emptyset \) for every distinct \( i , j \in I \).
If we can find \( y_i \in Y_i \) for all \( i \in I \), then we can extend this to a choice sequence \( y_n \in Y_n \subseteq X_n \) for all \( n \in \omega \) as required.
If \( I \) is finite, then the \( y_i \)s can be found without any appeal to choice.
If \( I \) is infinite, then \( I \approx \omega \) so we can find the \( y_i \)s by Proposition~\ref{prop:basicpropertiesofDC&AComega}\ref{prop:basicpropertiesofDC&AComega-d}.
\end{proof}

The following result follows from the argument of Theorem~\ref{th:Alessandro} together with Corollary~\ref{cor:anteAlessandro}.
\begin{corollary}
\( \FORALL{X} ( \DC ( X ) \IMPLIES \AComega^{< \omega } ( X )) \), where \( \AComega^{< \omega } ( X ) \) asserts that every countable collection of nonempty finite subsets of \(X\) has a choice function.
\end{corollary}

\subsection{Does \( \DC ( X ) \) imply \( \AComega ( X ) \)?}\label{sec:DCIAC}
By Theorem~\ref{th:folklore} and Theorem~\ref{th:Alessandro} 
\begin{equation}\label{eq:mainquestion}
\FORALL{X} ( \DC ( X ) \IMPLIES \AComega ( X ) ) 
\end{equation}
follows from either one of the following assumptions:
\begin{itemize}
\item
\( X \times 2 \precsim X \) for all infinite \( X \), 
\item
\( \AComega ( \R ) \).
\end{itemize}
Sageev in~\cite{Sageev:1975ff} proved that ``\( X \times 2 \precsim X \) for all infinite \( X \)'' does not imply \( \AComega ( \R ) \), while Monro in~\cite{Monro:1974aa} proved that \( \DC \) (and hence the weaker \( \AComega ( \R ) \)) does not imply ``\( X \times 2 \precsim X \) for all infinite \( X \)''.
So neither assumption implies the other.

The obvious question is if~\eqref{eq:mainquestion} is a theorem of \( \ZF \).
Suppose that there is a set \( X \) such that \( \DC ( X ) \wedge \neg \AComega ( X ) \).
By the proof of Lemma~\ref{lem:combinatorial} the set \( A \coloneqq F [ X ] \subseteq \Pow ( \omega ) \) is such that \( \DC ( A ) \) holds, as \( A \) is the surjective image of \( X \), and \( \AComega ( A ) \) fails, as otherwise, arguing as in Theorem~\ref{th:Alessandro}, \( \AComega ( X ) \) would hold.
Therefore if~\eqref{eq:mainquestion} fails, then the witness of this failure can be taken to be a subset of \( \R \).
In Section~\ref{sec:main} we construct a model of \( \ZF \) in which
\begin{equation}\label{eq:mainquestion2}
 \EXISTS{ A \subseteq \R} \left ( \DC ( A ) \wedge \neg \AComega ( A ) \right )
\end{equation}
showing that~\eqref{eq:mainquestion} is not a theorem of \( \ZF \).

By Lemma~\ref{lem:propertiesofA}\ref{lem:propertiesofA-c}, any \( A \) as in~\eqref{eq:mainquestion2} cannot contain a nonempty perfect set. 
Moreover, such a set \( A \) also needs to be Dedekind-infinite: indeed, \( A \) cannot be closed, as otherwise, by the usual Cantor-Bendixson argument, it would either be countable, contradicting \( \neg \AComega(A) \), or else it would contain a nonempty perfect set, which we already excluded; therefore \( A \) is not closed, and, by Lemma~\ref{lem:firstcountable}, \( A \) is Dedekind-infinite.

It can be shown that~\eqref{eq:mainquestion2} fails both in Cohen's first model (Proposition~\ref{prop:fcstatement}) and in the Feferman-Levy model (Proposition~\ref{prop:FefermanLevy}), and hence in both these models~\eqref{eq:mainquestion} holds.

\subsection{An equivalent formulation of \( \DC \)}
A \markdef{tree on \( X \)} is a \( T \subseteq {}^{< \omega } X \) that is closed under initial segments, that is if \( t \in T \) and \( s \subseteq t \) then \( s \in T \).
A tree \( T \) on \( X \) is \markdef{pruned} if for every \( t \in T \) there is \( s \in T \) such that \( t \subset s \).
A \markdef{branch} of \( T \) is a \( b \colon \omega \to X \) such that \( \FORALL{n \in \omega } ( b \restriction n \in T) \).
A tree \( T \) is \markdef{ill-founded} if it has a branch; otherwise it is \markdef{well-founded}.
Let
\begin{equation*}
\text{Any nonempty pruned tree on \( X \) is ill-founded} \tag{\( \DC _ \omega ( X ) \)}
\end{equation*}
and let \( \DC_ \omega \) be \( \FORALL{X} \DC_ \omega ( X ) \).
As \( \DC \) is equivalent to \( \DC_ \omega \) (Corollary~\ref{cor:versionsofDC} below), the axiom of Dependent Choice is often stated as \( \DC_\omega \).
The advantage of this formulation is that it can be generalized to ordinals larger than \( \omega \).

\begin{proposition}\label{prop:DCtrees}
\( \DC_\omega ( X ) \IFF \DC ( {}^{< \omega } X ) \), for every nonempty set \( X \).
\end{proposition}

\begin{proof}
(\( \Rightarrow \)) 
Suppose \( R \) is a binary relation on \( {}^{< \omega } X \) such that \( \FORALL{s} \EXISTS{t} ( s \mathrel{R} t ) \).
If \( \emptyset \mathrel{R} \emptyset \), then \( \seqLR{ \emptyset , \emptyset , \dots } \) is an \( R \)-chain as required, so we may assume otherwise.
Let \( R' \subseteq R \) be the sub-relation on \( {}^{< \omega } X \) defined by
\[
s \mathrel{R}' t \IFF s \mathrel{R} t \wedge \FORALL{t'} ( s \mathrel{R} t' \IMPLIES \lh ( t' ) \ge \lh ( t ) ) .
\]
The relation \( R' \) is total and any \( R' \)-chain is an \( R \)-chain.
Then 
\[
T = \setof{t \in {}^{< \omega } X}{ \EXISTS{s_0 , \dots , s_n } ( \emptyset \mathrel{R}' s_0 \mathrel{R}' \dots \mathrel{R}' s_n \AND t \subseteq s_0 \conc s_1 \conc \dots \conc s_n ) } 
\]
is a pruned tree on \( X \), so it has a branch.
By the minimality assumption of \( R' \), given a branch \( b \) of \( T \) one can construct inductively an \( R' \)-chain \( ( s_n )_{ n \in \omega } \) such that \( s_0 \conc s_1 \conc \dots \conc s_n \subseteq b \) for all \( n \).

\smallskip

(\( \Leftarrow \))
If \( T \) is a pruned tree on \( X \), let \( R \subseteq T \times T \) be defined by
\[
s \mathrel{R}t \IFF s \subset t \wedge \lh ( s ) + 1 = \lh ( t ) .
\]
As \( T \subseteq {}^{< \omega } X \) then \( \DC ( T ) \) holds, and since \( R \) is total, as \( T \) is pruned, there is an \( R \)-chain.
Any such chain yields a branch of \( T \).
\end{proof}

\begin{corollary}\label{cor:versionsofDC}
\( \DC \IFF \DC_ \omega \).
\end{corollary}

\begin{proposition}\label{prop:DComega}
Let \( X \) be a set.
\begin{enumerate-(a)}
\item\label{prop:DComegaa}
\( \DC_ \omega ( X ) \IMPLIES \DC ( X ) \).
\item\label{prop:DComegab}
\( \DC_ \omega ( X ) \IMPLIES \AComega ( X ) \).
\end{enumerate-(a)}
\end{proposition}

\begin{proof}
\( X \) injects into \( {}^{< \omega } X \), so part~\ref{prop:DComegaa} holds by Proposition~\ref{prop:DCtrees}.

For part~\ref{prop:DComegab} argue as follows.
If \( (A_n)_{ n \in \omega } \) is a sequence of nonempty subsets of \( X \), then \( \setof{ \seq{x_0 , \dots , x_{ n - 1 } } }{ \FORALL{i < n} ( x_i \in A_i ) } \) is a pruned tree on \( X \), and any branch of it is a sequence \( ( a_n )_{ n \in \omega } \) such that \( a_n \in A_n \) for all \( n \in \omega \).
\end{proof}

In light of Proposition~\ref{prop:DComega}, our main result, Theorem~\ref{th:Lorenzo}, tells us it is consistent with \( \ZF \) that there is a set \( A\subseteq \R \) for which \( \DC (A) \) holds but \( \DC_ \omega ( A ) \) fails.

\section{Symmetric extensions}\label{ssec:symmetricextensions}
The model we construct in Section~\ref{sec:main} is an iterated symmetric extension.
For the reader's convenience, let us recall a few facts about forcing and symmetric extensions. 

If \( \forcing{P} \) is a forcing notion, i.e.~a preordered set with a maximum \( 1_{ \forcing{P}} \), we convene that \( p \leq_{ \forcing{P}} q \) means that \( p \) is \markdef{stronger} than \( q \).
When there is no danger of confusion, we drop the subscript \( \forcing{P} \).
Dotted letters \( \dot{x}, \dot{y} , \dots \) vary over the class of \( \forcing{P} \)-names, \( \check{x} \) is the canonical \( \forcing{P} \)-name for \( x \), while \( \dot{G} \) is the \( \forcing{P} \)-name for the generic filter.
If \( F \) is a set of \( \forcing{P} \)-names, then, following Karagila's notation \cite{Karagila:2019aa}, \( F^\bullet \) is the \( \forcing{P} \)-name \( \setof{ ( \dot{x} , 1 ) }{ \dot{x} \in F } \). 
This notation extends naturally to ordered pairs and sequences, so \( ( \dot{x} , \dot{y} )^\bullet \coloneqq \setLR{ \set{ \dot{x}}^\bullet, \set{ \dot{x},\dot{y} }^\bullet }^\bullet \) and so on.
If \( G \subseteq \forcing{P} \) is \( \Vv \)-generic, then \( \dot{x}_G \) is the object in \( \Vv [ G ] \) obtained by evaluating \( \dot{x} \) with \( G \).

Let \( \forcing{P} \) be a forcing notion.
Every automorphism \( \pi \in \Aut ( \forcing{P} ) \) acts canonically on \( \forcing{P} \)-names as follows: given \( \dot{x} \) a \( \forcing{P} \)-name, 
\[
 \pi \dot{x} = \setof{ ( \pi \dot{y} , \pi p ) }{ ( \dot{y} , p ) \in \dot{x}}.
\]

\begin{lemma}[Symmetry Lemma, {\cite{Jech:2003pd}*{Lemma 14.37}}]
Let \( \forcing{P} \) be a forcing notion, \( \pi \in \Aut ( \forcing{P} ) \) and \( \dot{x}_1 , \dots , \dot{x}_n \) be \( \forcing{P} \)-names. 
For every formula \( \upvarphi ( x_1, \dots , x_n ) \)
\[
p \forces \upvarphi ( \dot{x}_1 , \dots , \dot{x}_n ) \IFF \pi p \forces \upvarphi ( \pi \dot{x}_1 , \dots , \pi \dot{x}_n ) .
\]
\end{lemma}

Let \( \mathcal{G} \) be a subgroup of \( \Aut ( \forcing{P}) \).
A nonempty collection \( \mathcal{F} \) of subgroups of \( \mathcal{G} \) is a \markdef{filter} on \( \mathcal{G} \) if it is closed under supergroups and finite intersections. 
A filter \( \mathcal{F} \) on \( \mathcal{G} \) is said to be \markdef{normal} if for every \( H \in \mathcal{F} \) and \( \pi \in \mathcal{G} \), the conjugated subgroup \( \pi H \pi^{-1} \) belongs to \( \mathcal{F} \) as well. 

We say that the triple \( \seq{ \forcing{P} , \mathcal{G} , \mathcal{F} } \) is a \markdef{symmetric system} if \( \forcing{P} \) is a forcing notion, \( \mathcal{G } \) is a subgroup of \( \Aut ( \forcing{P} ) \) and \( \mathcal{F} \) is a normal filter on \( \mathcal{G} \). 
Given a \( \forcing{P} \)-name \( \dot{x} \), we say that \( \dot{x} \) is \markdef{\( \mathcal{F} \)-symmetric} if there exists \( H \in \mathcal{F} \) such that for all \( \pi \in H \), \( \pi \dot{x} = \dot{x} \). 
This definition extends by recursion: \( \dot{x} \) is \markdef{hereditarily \( \mathcal{F} \)-symmetric}, if \( \dot{x} \) is \( \mathcal{F} \)-symmetric and every name \( \dot{y} \in \dom ( \dot{x} ) \) is hereditarily \( \mathcal{F} \)-symmetric. 
We denote by \( \HS _\mathcal{F} \) the class of all hereditarily \( \mathcal{F} \)-symmetric names.

\begin{theorem}[{\cite{Jech:2003pd}*{Lemma 15.51}}]
Suppose that \( \seq{ \forcing{P} , \mathcal{G} , \mathcal{F} } \) is a symmetric system and \( G \subseteq \forcing{P} \) is a \( \Vv \)-generic filter. 
Denote by \( \mathcal{N} \) the class \( \setof{ \dot{x}_G }{ \dot{x} \in \HS _\mathcal{F}} \), then \( \mathcal{N} \) is a transitive model of \( \ZF \), and \( \Vv \subseteq \mathcal{N} \subseteq \Vv [ G ] \).
\end{theorem}

The class \( \mathcal{N} \) is also known as a \markdef{symmetric extension} of \( \Vv \). 
Symmetric extensions are often used to produce models of \( \ZF \) in which the axiom of choice fails. 
We focus on this notion by discussing the construction due to Cohen of a symmetric extension in which there is an infinite, D-finite set of reals. 
This model will be the first step in the forcing iteration in Theorem~\ref{th:Lorenzo}. 

\subsection{The first Cohen model}\label{subsec:firstCohen}

Let \( \forcing{P}_0 \) be the forcing that adds countably many Cohen reals, i.e.
\[
\forcing{P}_0 = \setofLR{ p }{ \EXISTS{I \subseteq \omega } ( p \colon I \to {}^{<\omega} 2 \text{, and \( I \) is finite}) } ,
\]
with \( p \leq q \) if \( \dom ( p ) \supseteq \dom ( q ) \) and \( p ( n ) \supseteq q ( n ) \) for all \( n \in \dom ( q ) \). 
Although this is not the standard presentation of such a forcing, this way of defining \( \forcing{P}_0 \) will become useful in the Section~\ref{sec:main}. 
Let \( \dot{a}_n \) be the canonical name for the \( n \)-th Cohen real, that is
\begin{equation}\label{eq:nameforCohenreal}
\dot{a}_n = \setof{ ( \check{ ( k , i ) } , p ) }{ p \in \forcing{P}_0 \wedge n \in \dom p \wedge p ( n ) ( k ) = i } . 
\end{equation}
Observe that \( \dot{A} \coloneqq \setof{ \dot{a}_n }{ n \in \omega } ^\bullet \) is forced to be a dense subset of \( {}^ \omega 2 \).

Every permutation \( \pi \) on \( \omega \) induces an automorphism of \( \forcing{P}_0 \) as follows: given \( p \in \forcing{P}_0 \), we let \( \pi p \in \forcing{P}_0 \) be defined by 
\[
\FORALL{ n \in \dom ( p ) } \bigl ( \pi p ( \pi n ) = p ( n ) \bigr ) .
\]
We conflate the notation by using the same symbol \( \pi \) to denote both the permutation and the automorphism on \( \forcing{P}_0 \) it induces. 
Let \( \mathcal{G}_0 \) be the group of all such automorphisms. 
For every finite \( E \subset \omega \), let \( \Fix ( E ) \) be the subgroup of \( \mathcal{G}_0 \) of all those automorphisms induced by permutations that pointwise fix the set \( E \). 
Let \( \mathcal{F}_0 \) be the filter on \( \mathcal{G}_0 \) generated by \( \setof{ \Fix (E) }{ E \subset \omega \text{ finite} } \). 
It is easy to check that \( \mathcal{F}_0 \) is actually a normal filter on \( \mathcal{G}_0 \), and hence \( \seq{ \forcing{P}_0 , \mathcal{G}_0 , \mathcal{F}_0 } \) is a symmetric system. 
Let \( G \) be a \( \Vv \)-generic filter for \( \forcing{P}_0 \), and let \( \mathcal{N}_0 \) be the corresponding symmetric extension, which we call \markdef{first Cohen model}. 

Denote by \( A \) the realization of the name \( \dot{A} \) in \( \Vv [ G ] \), i.e.~the set \( \dot{A}_G \). 
Note that every \( \dot{a}_n \) is in \( \HS _{\mathcal{F}_0 } \) and so is \( \dot{A} \).

\begin{proposition}[{\cite{Jech:2003pd}*{Example 15.52}}]\label{prop:FirstCohen}
\( \mathcal{N}_0 \Models ``A \) is D-finite".
\end{proposition}

In \( \mathcal{N}_0 \), the set \( A \), being infinite and D-finite, is certainly not separable as a subspace of \( \R \) ---indeed, every infinite, separable \( \mathrm{T}_1 \) space is Dedekind-infinite. Moreover, \( \DC ( A ) \) also fails, as otherwise \(A\) would be Dedekind-infinite (see the penultimate paragraph of Section~\ref{sec:DCIAC}). 

The simultaneous local failure of both \( \AComega \) and \( \DC \) is not accidental---the next proposition shows that, in the first Cohen model, any \( X \) such that \( \DC ( X ) \) holds must be well-orderable, and hence \( \AComega ( X ) \) holds.

\begin{proposition}\label{prop:fcstatement}
\( \mathcal{N}_0 \Models \forall X \left ( \DC ( X ) \IMPLIES \AC (X) \right ) \).
\end{proposition}

\begin{lemma}\label{lem:familyoffinitesets}
Let \( X \) be a linearly ordered set, and let \( Y \subseteq [ X ]^{< \omega } \).
If \( \omega \precsim Y \), then \( \omega \precsim \bigcup Y \).
\end{lemma}

\begin{proof}
Let \( \leq \) be a linear ordering of \( X \), and let \( ( Z_n )_{n \in \omega}\) be a sequence of distinct elements of \( Y \). 
By passing to a subsequence, we may assume that \( Z_{ n + 1 } \nsubseteq Z_0 \cup \dots \cup Z_n \), and that \( Z_0 \neq \emptyset \).
Let \( z_0 \) be the \( \le \)-least element of \( Z_0 \), and \( z_{ n + 1} \) be the \( \le \)-least element of \( Z_{n + 1} \setminus ( Z_0 \cup \dots \cup Z_n ) \) for every \( n \).
The \( z_n \)s are distinct, and belong to \( \bigcup Y \), as required.
\end{proof}

\begin{lemma}\label{lem:madc}
If \( \DC ( Y ) \) with \( Y \subseteq [ \R ]^{<\omega} \) infinite, then \( \omega \precsim \bigcup Y \).
\end{lemma}

\begin{proof}
If \( \bigcup Y \) has no limit points, then it is discrete, so \( \omega \precsim \bigcup Y \). 
Suppose otherwise, and let \( x \in \R \) be a limit point of \( \bigcup Y \). 
It is enough to show that \( \omega \precsim Y \) and then apply Lemma~\ref{lem:familyoffinitesets} with \( X = \R \). 
Without loss of generality, we may assume that \( \set{ x} , \emptyset \notin Y \).
For all \( Z \in Y \) let \( d ( x , Z ) = \min \setof{ | r - x | }{ r \in Z \setminus \set { x } } \) be the distance of \( x \) from the rest of \( Z \).
Let \( R \subseteq Y ^2 \) be the binary relation defined as follows: for every \( Z , W \in Y \), 
\[ 
R ( Z , W ) \IFF d ( x , W ) < d ( x , Z ) .
\]
The relation \( R \) is acyclic, and, by our hypothesis on \( x \), it is total. 
It follows from \( \DC ( Y ) \) that \( R \) has an infinite chain, and hence \( \omega \precsim Y \).
\end{proof}

\begin{proof}[Proof of Proposition~\ref{prop:fcstatement}]
In the first Cohen model, for every set \( X \), there is a map \( s_X \colon X \to [ A ]^{<\omega} \), known as the least support map, such that \( s^{-1} ( \set{ B } ) \) is well-orderable for every \( B \in [ A ]^{ < \omega} \)~\cite{Jech:1973gf}*{Theorem 5.21, Lemma 5.25}.

Let \( X \) be such that \( \DC ( X ) \) holds. 
Then also \( \DC ( \ran ( s_X ) ) \) holds. 
If \( \ran ( s_X ) \) were infinite, then letting \( Y= \ran ( s_X ) \) in Lemma~\ref{lem:madc} we would have \( \omega \precsim \bigcup \ran ( s_X ) \subseteq A \), against the fact that \( A \) is D-finite. 
Hence \( \ran ( s_X ) \) is finite, and \( X \), being a finite union of well-orderable sets, is well-orderable.
\end{proof}

\section{The main result}\label{sec:main}
This section is devoted to proving the following:

\begin{theorem}\label{th:Lorenzo}
It is consistent with \( \ZF \) that there is a set \( A \subseteq \R \) such that \( \DC ( A ) \) and \( \neg \AComega ( A ) \).
\end{theorem}

\subsection{Outline of the proof}
We prove the theorem via an iteration of symmetric extensions of length \( \omega \).
We start the iteration with the first Cohen model \( \mathcal{N}_0\), with \( A \in \mathcal{N}_0 \) being the generic D-finite set of reals (see Section~\ref{subsec:firstCohen}). 
As noted right after Proposition~\ref{prop:FirstCohen}, in \( \mathcal{N}_0 \) the set \( A \) is not separable (in particular \( \AComega ( A ) \) fails) and \( \DC ( A ) \) fails.
Next, we define a chain of models \( \mathcal{N}_0 \subset \mathcal{N}_1 \subset \dots \subset \mathcal{N}_\omega \) such that, for each \( n \), \( \mathcal{N}_{n+1} \) is a symmetric extension of \( \mathcal{N}_n \) that contains a generic set of chains for all binary relations in \( \mathcal{N}_n \) that are total and acyclic on \( A \). 
At the final stage, \( \mathcal{N}_\omega \), which is our model, is going to be something resembling ``the model of sets definable from finitely many elements of \( \bigcup_n \mathcal{N}_n \)". 
If we do the construction properly, we can prove that in \( \mathcal{N}_\omega \) we have added enough countable subsets of \( A \) (or, equivalently, enough sequences over \( A \)) to guarantee \( \DC ( A ) \) (Theorem~\ref{th:dc}), but \( A \) is still not separable, in particular \( \AComega ( A ) \) fails (Corollary~\ref{cor:separable}).
 
Actually, we don't only show that \( A \) is not separable in our model, but we give a topological characterization of its separable subsets: among the subsets of \( A \), the separable ones are precisely those which are scattered with finite scattered height (Definition~\ref{def:scattered}, Theorem~\ref{th:separablesub}).

\subsection{The symmetric system}
We define recursively a sequence \( \seq{ \forcing{P}_n , \mathcal{G}_n , \mathcal{F}_n }_{ n \in \omega } \) of symmetric systems. 
Let \( \seq{ \forcing{P}_0 , \mathcal{G}_0 , \mathcal{F}_0 } \) be the symmetric system defined in Section~\ref{subsec:firstCohen}, i.e.~the one that induces the first Cohen model. 
For each \( n \) we denote by \( \leq _n, \forces _n \) the ordering and the forcing relation of \( \forcing{P}_n \), respectively, and by \( \HS _n \) the class \( \HS _{ \mathcal{F}_n } \), i.e.~the class of all hereditarily \( \mathcal{F}_n \)-symmetric \( \forcing{P} _n \)-names. 
We also let 
\begin{equation}\label{eq:R_n}
\mathcal{R}_n = \setof{ \dot{R} \in \HS _n }{ \FORALL{ \dot{x} \in \dom ( \dot{R} ) } \EXISTS{ m , k } \in \omega \ ( \dot{x} = ( \dot{a}_{ m} , \dot{a}_{k} )^\bullet ) }, 
\end{equation}
where the \( \dot{a}_i \)s are as in~\eqref{eq:nameforCohenreal}.
Then \( \mathcal{R}_n \) is the set of all ``good" hereditarily \( \mathcal{F}_n \)-symmetric \( \forcing{P} _n \)-names for binary relations on \( \dot{A} \). 
 
Recursively on \( n \), we define \( \forcing{P} _{n+1} \) to be the set of all the sequences \( p = \seqof{ p_k }{ k \leq n + 1 } \) such that 
\begin{enumerate}[label={(\arabic*)}]
 \item 
 \( p \restriction n + 1 \in \forcing{P} _{n} \),
 \item 
\( p_{n+1} \colon \dom ( p_{n + 1 } ) \to \mathcal{R}_n \times {}^{<\omega} \omega \) with \( \dom ( p_{n+1} ) \) a finite subset of \( \omega \),
 \item 
 For each \( k \in \dom ( p_{n + 1 } ) \) with \( p_{n+1} ( k ) = ( \dot{R}, \seq{ m_0 , \dots, m_h } ) \) we have 
\[ 
 p \restriction n + 1 \forces _n ``\dot{R} \text{ is total, acyclic and }\dot{a}_{ m_0 } \mathrel{\dot{R}} \dot{a}_{ m_1 } \mathrel{\dot{R}} \dots \mathrel{\dot{R}} \dot{a}_{ m_h }",
\] 
\end{enumerate}
where, at stage \( n = 0 \), we identify the conditions \( p \in \forcing{P} _0 \) with their singleton sequence \( \seq{ p } \).
 
For each \( p \in \forcing{P} _{n+1} \) and \( k \in \dom ( p_{ n + 1 } ) \) with \( p_{ n + 1 } ( k ) = ( \dot{R} , s ) \), we denote \( \dot{R} \) and \( s \) by \( p_{n + 1}^R ( k ) \) and \( p_{n + 1}^s ( k ) \), respectively. 
Given \( p , q \in \forcing{P} _{ n + 1 } \) we let \( p \leq _{ n + 1 } q \) if and only if
\begin{itemize}
 \item 
 \( p \restriction n + 1 \leq _n q \restriction n + 1 \),
 \item 
 \( \dom ( p_{n + 1} ) \supseteq \dom ( q_{n + 1 } ) \),
 \item 
 \( \FORALL{ k \in \dom ( q_{n + 1 } ) } ( p_{n + 1 }^R ( k ) = q_{ n + 1 }^R ( k ) \) and \( p_{ n + 1 }^s ( k ) \supseteq q_{n + 1 }^s ( k )) \).
\end{itemize}
This defines the forcing \( \forcing{P} _{n + 1} \). 
Now we are left to define the subgroup \( \mathcal{G}_{ n + 1 } \) of \( \Aut ( \forcing{P} _{n + 1 } ) \), and the filter \( \mathcal{F}_{ n + 1 } \).

Consider a sequence \( \vec{\pi} = \seq{ \pi_0, \dots, \pi_{ n + 1 } } \) with each \( \pi_i \) being a permutation of \( \omega \). 
By induction hypothesis
\footnote{ At \( n = 0 \) we identify each \( \pi \in \mathcal{G}_0 \) with the singleton sequence \( \seq{ \pi } \).}
\[ 
\vec{\rho} \coloneqq \seq{ \pi_0, \dots, \pi_n } = \vec{\pi} \restriction n + 1 
\]
induces an automorphism \( \vec{ \rho } \in \mathcal{G}_n \). 
Note that, as in Section~\ref{subsec:firstCohen}, we conflate the notation by using the same symbol to denote both sequences of permutations and the automorphisms they induce. 
Now, the sequence \( \vec{\pi} \) induces an automorphism on \( \forcing{P} _{ n + 1 } \) as follows: given \( p \in \forcing{P} _{ n + 1 } \), we let \( \vec{\pi} p \) be the condition in \( \forcing{P} _{ n + 1 } \) such that \( ( \vec{\pi} p ) \restriction n + 1 = \vec{\rho } ( p \restriction n + 1) \)
and, for each \( k \in \dom ( p_{ n + 1 } ) \) with \( p_{ n + 1 }^s ( k ) = \seq{ m_0 , \dots , m_h } \) and \( p_{ n + 1 }^R ( k ) = \dot{R} \),
\begin{align*}
 ( \vec{\pi} p )_{ n + 1 }^R ( \pi_{ n + 1 } ( k ) ) & = \vec{\rho} ( \dot{R} ) ,
 \\
 ( \vec{\pi} p )_{ n + 1 }^s ( \pi_{ n + 1 } ( k ) ) & = \seq{ \pi_0 ( m_0 ) , \dots, \pi_0 ( m_h ) } .
\end{align*}

Let \( \mathcal{G}_{n + 1} \) be the group of all such automorphisms on \( \forcing{P} _{n + 1} \), i.e.~the ones induced by sequences (of length \( n + 2 \)) of permutations of \( \omega \). 
For each sequence \( \vec{H} = \seq{ H_0 , \dots , H_{ n + 1 } } \) of subsets of \( \omega \), we let \( \Fix ( \vec{H} ) \) be the subgroup of all those \( \vec{\pi} \in \mathcal{G}_{n + 1} \) such that \( \pi_k \) pointwise fixes \( H_k \) for all \( k \leq n + 1 \). 
We define \( \mathcal{F}_{n + 1} \) as the filter on \( \mathcal{G}_{n + 1} \) generated by \( \setof{ \Fix ( \vec{H} ) }{ H_k \text{ is finite for all } k \leq n +1 } \).
From now on, we use the symbol \( \vec{H} \) to denote  \emph{finite sequences of finite} subsets of \( \omega \).

This ends the inductive definition of the sequence \( \seq{ \forcing{P} _n , \mathcal{G}_n , \mathcal{F}_n }_{ n \in \omega } \). 
Note that, for each \( n < m \), there is a natural complete embedding \( i_{ n , m } \colon \forcing{P} _n \to \forcing{P} _m \) and a natural embedding \( j_{ n , m } \colon \mathcal{G}_n \to \mathcal{G}_m \). 
Thus we let \( \forcing{P} \) and \( \mathcal{G} \) be the direct limits of the forcings \( \forcing{P} _n \) and of the groups \( \mathcal{G}_n \), respectively. 

We now define the normal filter \( \mathcal{F} \) on \( \mathcal{G} \) in the expected way: we let \( \mathcal{F} \) be the filter generated by 
\[ 
\setof{ \Fix ( \vec{H} ) }{ H_k \text{ is finite for all } k < \lh ( \vec{H} ) } ,
\] 
where, given any \( \vec{H} \), \( \Fix ( \vec{H} ) \) is the subgroup of \( \mathcal{G} \) made of all those \( \vec{\pi} \) such that \( \pi_k \) pointwise fixes \( H_k \) for all \( k < \lh ( \vec{H} ) \).

A condition \( p \) of the direct limit \( \forcing{P} \) is a finite sequence \( \seq{ p_0, \dots, p_n } \), and it is identified with \( \seq{ p_0, \dots, p_n , \emptyset , \dots , \emptyset } \) and with \( \seq{ p_0, \dots, p_n, \emptyset, \emptyset \dots } \), that is a sequence obtained by concatenating \( p \) with a finite sequence of empty sets or with the infinite sequence of empty sets. 
We treat analogously the \( \vec{H} \)s.

Henceforth \( \seq{ \forcing{P} , \mathcal{G}, \mathcal{F}} \) is our symmetric system, with \( \HS \) being the class of all \( \mathcal{F} \)-symmetric \( \forcing{P} \)-names and \( \leq , \forces \) being the ordering and the forcing relation of \( \forcing{P} \), respectively.
\begin{remark}
Our iterative construction fits into the general framework developed by Asaf Karagila~\cite{Karagila:2019aa} to deal with iterations of symmetric extensions.
\end{remark}

Given an \( \dot{x} \in \HS \), we say that \( \vec{H} \) is a \markdef{support} of \( \dot{x} \) if \( \vec{\pi} \dot{x} = \dot{x} \) for all \( \vec{\pi} \in \Fix ( \vec{H} ) \). 
Also, given \( p = \seq{ p_0, \dots, p_{n} } \in \forcing{P} \) and \( \vec{H} = \seq{ H_0, \dots , H_n } \), we write \( p \restriction \vec{H} \) to denote the sequence \( \seq{ p_0 \restriction H_0 , \dots, p_{n} \restriction H_{n} } \). 
Note that the latter sequence needs not to belong to \( \forcing{P} \). 

\begin{lemma}[Restriction Lemma]
Let \( \upvarphi (x_1, \dots, x_n) \) be a formula in the forcing language, and let \( \dot{x}_1,\dots, \dot{x}_n \in \HS \). 
For any \( p \in \forcing{P} \) and for any \( \vec{H} \), if \( \vec{H} \) is a support for each of the \( \dot{x}_i \)'s and, for all \( m > 0 \), for all \( k \in H_m \cap \dom (p_m) \), \( \vec{H} \restriction m \) is a support for \( p^R_m ( k ) \) and \( \ran ( p^s_m ( k ) ) \subseteq H_0 \), then \( p \restriction \vec{H} \in \forcing{P} \) and 
\[
p \forces \upvarphi ( \dot{x}_1, \dots, \dot{x}_n) \IFF p \restriction \vec{H} \forces \upvarphi ( \dot{x}_1, \dots, \dot{x}_n).
\]
\end{lemma}

\begin{proof}
We prove the lemma by induction on the length of \( \vec{H} \). 

Let's first assume \( \vec{H} = \seq{ H_0 } \) for some finite \( H_0\subset \omega \), then \( p \restriction \vec{H} \in \forcing{P} _0 \). 
Assume for a contradiction that \( p \restriction \vec{H} \not \forces \upvarphi ( \dot{x}_1, \dots, \dot{x}_n ) \), then there is a \( q \leq p \restriction \vec{H} \) such that \( q \forces \neg \upvarphi ( \dot{x}_1, \dots, \dot{x}_n ) \). 
Let \( \vec{\pi} \in \mathcal{G} \) such that \( \pi_0 \) pointwise fixes \( H_0 \) and such that \( \pi_0[ \dom (q_0)] \cap \dom (p_0) = H_0 \cap \dom (p_0) \) and \( \pi_m[ \dom (q_m)] \cap \dom (p_m) = \emptyset \) for all \( m > 0 \). 
In particular, \( \vec{\pi} \in \Fix ( \vec{H} ) \). 
By hypothesis, \( \vec{H} \) is a support for all the \( \dot{x}_i \)'s. 
Thus, by the Symmetry Lemma, \( \vec{\pi} q \forces \neg \upvarphi ( \dot{x}_1, \dots, \dot{x}_n ) \). 
However, \( p \) and \( \vec{\pi} q \) are compatible, contradiction.

Now let's assume that \( \vec{H} = \seq{ H_0, \dots, H_m } \).

\begin{claim}
\( p \restriction \vec{H} \in \mathbf{P}_m \).
\end{claim}
\begin{proof}
By induction hypothesis, \( p \restriction ( \vec{H} \restriction m ) \in \mathbf{P}_{ m - 1 } \). 
Fix a \( k \in H_m \cap \dom (p_m) \). 
Let \( \dot{R} = p_m^R ( k ) \) and \( \seq{ n_0 , \dots , n_h } = p_m^s ( k ) \). 
Then, by assumption, \( \vec{H}\restriction m \) is a support of \( \dot{R} \) and \( n_i \in H_0 \) for every \( i \le h \), or, equivalently, \( H_0 \) is a support for \( \dot{a}_{n_i} \). 
By definition of \( \mathbf{P}_m \), 
\[
p \restriction m \forces ``\dot{R} \text{ is total, acyclic and }\dot{a}_{n_0} \mathrel{\dot{R}} \dot{a}_{n_1} \mathrel{\dot{R}} \dots \mathrel{\dot{R}} \dot{a}_{n_h}".
\]
By induction hypothesis,
\[
( p \restriction \vec{H} ) \restriction m = p \restriction ( \vec{H} \restriction m ) \forces ``\dot{R} \text{ is total, acyclic and }\dot{a}_{n_0} \mathrel{\dot{R}} \dot{a}_{n_1} \mathrel{\dot{R}} \dots \mathrel{\dot{R}} \dot{a}_{n_h}".
\]
Therefore, \( p \restriction \vec{H} \in \mathbf{P}_m \).
\end{proof}

Assume for a contradiction that \( p \restriction \vec{H} \not \forces \upvarphi ( \dot{x}_1, \dots, \dot{x}_n ) \), then there is a \( q \leq p \restriction \vec{H} \) such that \( q \forces \neg \upvarphi ( \dot{x}_1, \dots, \dot{x}_n ) \). 
Let \( \vec{\pi} \in \mathcal{G} \) such that \( \pi_l \) pointwise fixes \( H_l \) for each \( l \leq m \) and such that \( \pi_l [ \dom ( q_l ) ] \cap \dom ( p_l ) = H_l \cap \dom ( p_l ) \) for all \( l \leq m \) and \( \pi_l [ \dom ( q_l ) ] \cap \dom ( p_l ) = \emptyset \) for all \( l > m \). 
In particular, \( \vec{\pi} \in \Fix ( \vec{H} ) \). 
Thus, by the Symmetry Lemma, \( \vec{\pi} q \forces \neg \upvarphi ( \dot{x}_1, \dots, \dot{x}_n ) \). 

\begin{claim}
\( p \) and \( \vec{\pi} q \) are compatible.
\end{claim}
\begin{proof}
It suffices to show that \( p_l^R ( k ) = ( \vec{\pi} q )_l^R ( k ) \) and that the sequence \( p_l^s ( k ) \) is extended by \( ( \vec{\pi} q)_l^s ( k ) \), for every \( l \le m \) and for every \( k \in \dom ( p_l ) \cap \dom ( ( \vec{ \pi } q )_l ) \). 
Note that \( \dom ( ( \vec{\pi} q )_l ) = \pi_l [ \dom ( q_l ) ] \). 
Fix an \( l \le m \) and a \( k \in \dom ( p_l ) \cap \pi_l [ \dom ( q_l )] \). 
By the way we chose \( \vec{\pi} \), we must have \( k \in H_l \). 
Also, as we assumed \( q \le p \), we have \( q_l^R ( k ) = p_l^R ( k ) \) and \( q_l^s ( k ) \supseteq p_l^s ( k ) \). 
Moreover, we assumed \( \vec{H} \restriction l \) to be a support for \( p_l^R ( k ) \), and we have picked \( \vec{\pi} \) so that \( \vec{\pi} \in \Fix ( \vec{H} ) \). 
In particular, \( (\vec{\pi} \restriction l) \, ( p_l^R ( k ) ) = p_l^R ( k ) \) and \( \pi_l ( k ) = k \).
Therefore, by the definition of the induced automorphism \( \vec{\pi} \in \mathcal{G} \), 
\[
 ( \vec{ \pi } q )_l^R ( k ) = (\vec{\pi} \restriction l) \, ( q_l^R ( \pi_l^{-1} ( k ) ) ) = (\vec{\pi} \restriction l) \,( q_l^R ( k ) ) = (\vec{\pi} \restriction l) \, ( p_l^R ( k ) ) = p_l^R ( k ) .
\]
Moreover, since we assumed \( \ran ( p_l^s ( k ) ) \subseteq H_0 \), and \( \pi_0 \in \Fix ( H_0 ) \), we have, for every \( i< \lh ( p_l^s ( k ) ) \),
\[
 ( \vec{\pi} q )_l^s ( k ) ( i ) = \pi_0 ( q_l^s ( \pi_l^{-1} ( k ) ) ( i ) ) = \pi_0 ( q_l^s ( k ) ( i ) ) = \pi_0 ( p_l^s ( k ) ( i ) ) = p_l^s ( k ) ( i ) ,
\]
and therefore \( ( \vec{\pi} q )_l^s ( k ) \supseteq p_l^s ( k ) \).
\end{proof}
As before, the fact that \( p \) and \( \vec{\pi} q \) are compatible yields the desired contradiction and concludes the proof.
\end{proof}

\subsection{The model}
For each \( n , k \in \omega \), we let
\begin{align*}
 \dot{f}_{ n , k } &= \setof{ ( ( \check{l} , \dot{a}_m )^\bullet , p ) }{ l , m \in \omega \AND p \in \forcing{P}_{ n + 1 } \AND p^s_{ n + 1 } ( k ) ( l ) = m } , 
 \\ 
\dot{F}_n & = \setof{ \dot{f}_{ n , k } }{ k \in \omega }^\bullet.
\end{align*}
Note that these \( \forcing{P} \)-names are in \( \HS \). 
These, together with \( \dot{A}\) and the \( \dot{a}_n \)s, are the (hereditarily symmetric) names for all the generic sets we are interested in.
Observe that \( \dot{f}_{ n , k } \) is a \( \forcing{P}_{ n + 1 } \)-name for an \( R \)-chain belonging to \( \mathcal{N}_{ n + 1} \), where \( R \) is the relation with \( \forcing{P}_{n+1} \)-name \( \setof{ ( p^R_{ n + 1 } ( k ) , p ) }{ p \in \forcing{P}_{ n + 1 } } \).

Fix a \( \Vv \)-generic filter \( G \) for \( \forcing{P} \) and, for all \( n \), let \( \mathcal{N}_n \) be the symmetric extension obtained from \( \seq{ \forcing{P} _n , \mathcal{G}_n , \mathcal{F}_n } \), and \( \mathcal{N} \) be the symmetric extension, obtained from \( \seq{ \forcing{P} , \mathcal{G}, \mathcal{F} } \). 
Clearly we have 
\[
 \Vv \subseteq \mathcal{N}_0 \subseteq \mathcal{N}_1 \subseteq \dots \subseteq \mathcal{N} = \mathcal{N}_\omega \subseteq \Vv [G] .
\] 
For each \( \forcing{P} \)-name (e.g.~\( \dot{A} \)), we let its symbol without the dot (i.e.~\( A \)) be its evaluation according to \( G \) (i.e.~\( \dot{A}_G \)).

\begin{lemma}\label{lem:newchains}
 For every \( n \in \omega \), for every total and acyclic binary relation \( R \in \mathcal{N}_n \) on \( A \), there is an an \( R \)-chain in \( \mathcal{N}_{n+1} \).
\end{lemma}
\begin{proof}

Let \( p \in G \) and \( \dot{R} \in \HS _n \) such that 
\[
p \forces \dot{R}\subseteq \dot{A} \times \dot{A} \text{ total and acyclic,}
\]
and without loss of generality we may assume that \( p \in \forcing{P} _n \). 
We show that \( \dot{R}_G = \dot{S}_G \) for some \( \dot{S} \in \mathcal{R}_n \) as in~\eqref{eq:R_n}.
Let
\[ 
\dot{S} = \setof{ ( ( \dot{a}_m , \dot{a}_k )^\bullet , q ) }{ m , k \in \omega, q \in \forcing{P} _n \text{, and } q \forces \dot{a}_m \mathrel{\dot{R}} \dot{a}_k } .
\] 
It readily follows that \( \dot{S} \) is in \( \mathcal{R}_n \) and \( p \forces \dot{R} = \dot{S} \). 
Fix any \( q \in \forcing{P}_{n + 1} \) with \( q \le p \).
Pick an \( m \in \omega \setminus \dom ( q_{ n + 1 } ) \) and consider the finite sequence \( q' \) such that \( q'_l = q_l \) for every \( l \neq n + 1 \) and \( q'_{ n + 1 } = q_{ n + 1 } \cup \set{ ( m , ( \dot{S} , \emptyset ) ) } \). 
Then \( q' \in \forcing{P}_{ n + 1 } \), \( q' \le q \) and 
\[
q' \forces \dot{f}_{ n , m } \text{ is an } \dot{S} \text{-chain, and } \dot{S} = \dot{R} .
\] 
By density, 
\[
p \forces \exists f \in \dot{F}_n \text{ which is an } \dot{R} \text{-chain}.
\]
Since \( F_n \in \mathcal{N}_{ n + 1 } \) we are done.
\end{proof}

Since \(A\) is not closed as a set of reals, there exists a total and acyclic binary relation over \(A\) in \(\mathcal{N}_0\) (see Lemma~\ref{lem:firstcountable}). Therefore, by Lemma~\ref{lem:newchains}, the set \(A\) becomes Dedekind-infinite already in \(\mathcal{N}_1\). However, the next proposition tells us that the range of any generic chain introduced by the iteration is far from being dense in \(A\). This result is crucial in showing that \(A\) is not separable in \(\mathcal{N}\).

In order to get to the key proposition, we need to recall the well-known notion of scattered space (e.g.~see~\cite{MR296671}*{\S 8.5}).

\begin{definition}\label{def:scattered}
Given a topological space \( X \), we let by ordinal induction
\begin{align*}
 X^{ ( 0 ) } &= X,
 \\ 
 X^{ ( \alpha + 1 ) } &= \setof{ x \in X^{ ( \alpha ) } }{ x \text{ is a limit point of } X^{ ( \alpha ) } }, 
 \\
 X^{ ( \lambda ) } &= \bigcap_{\alpha < \lambda} X^{ ( \alpha ) } \qquad \text{for \( \lambda \) a limit ordinal}.
\end{align*}
For every space \( X \) there is necessarily an ordinal \( \alpha \) such that \( X^{ ( \alpha ) } = X^{ ( \alpha + 1 ) } \), and we call the least such ordinal the \markdef{scattered height} of the space. 
A topological space \( X \) is \markdef{scattered} if there is an \( \alpha \) such that \( X^{ ( \alpha ) } = \emptyset \).
\end{definition}

It is easy to check in \( \ZF \) that every second countable scattered space is countable~\cite{MR296671}*{Proposition 8.5.5}.
If \( \emptyset \neq X \subseteq \R \) is dense in itself, then \( X^{( 1 )} = X \) and \( X \) is not scattered.
In particular, \( A \) is not scattered.

For each \( t \in {}^{<\omega} 2 \), we denote by \( \dot{\Nbhd}_t \) the canonical name for the basic open set \( \Nbhd_{t} \), i.e.~the set of all infinite binary sequences extending \( t \).

\begin{proposition}\label{lem:core}
For each \( n , k \in \omega \), \( \mathcal{N} \Models \left ( \Cl_A ( \ran ( f_{ n , k } ) ) \right )^{ ( n + 2 ) } = \emptyset \).
\end{proposition}

\begin{proof}
In other words, we want to show that in \( \mathcal{N} \) (or, equivalently
\footnote{Note that the formula \( \upvarphi ( x , y , \alpha ) \coloneqq ``\alpha\) is an ordinal, \( x \subseteq y \) are sets of reals and \( \Cl_y ( x ) \) is scattered of height \( \le \alpha" \) is a \( \Delta_1^{\ZF} \)-formula. 
In particular, it is absolute between models of \( \ZF \).}%
, in \( \Vv [ G ] \)), for every \( n , k \), the closure with respect to \(A\) of the range of \( f_{ n , k } \) is scattered of height at most \( n + 2 \). 
For any \( H \subseteq \omega \), let us introduce the \( \forcing{P} \)-name 
\begin{equation}\label{eq:dotA_H}
 \dot{A}_H \coloneqq \setof{\dot{a}_m }{ m \in H }^\bullet .
\end{equation}

Let \( ( \dagger ) \) be the statement \( \FORALL{  n \in \omega }  ( \dagger )_n \), where \( ( \dagger )_n \) is the following statement:
\begin{equation*}\label{eq:mainprop}\tag*{\( ( \dagger )_n \)}
\begin{minipage}[c]{0.85\textwidth}
Let \( k \in \omega \), \( p = \seq{ p_0, \dots p_n } \in \forcing{P} _n \), \( \dot{R} \in \mathcal{R}_n \) with support \( \vec{H} = \seq{ H_0, \dots, H_n } \) such that \( p \forces \dot{R} \) is total and acyclic". 
Assume also that, for each \( i \leq n \), \( \dom ( p_i ) = H_i \), and, for all \( 0 < i \leq n \), for all \( j \in H_i \), \( \vec{H} \restriction i \) is a support for \( p_{i}^R ( j ) \). 
Then 
\[
p \conc \seq{\set{ ( k , ( \dot{R}, \emptyset ) ) } } \forces \bigl ( \Cl_{\dot{A}} ( \ran ( \dot{f}_{ n , k } ) ) \bigr )^{ ( 1 ) } \subseteq \dot{A}_{H_0} \cup \bigcup_{\smash[b]{\substack{i < n \\ j \in H_{ i + 1 } } } } \ran ( \dot{f}_{ i , j } ) . 
\]
\end{minipage} 
\end{equation*}

\begin{remarks*}
\begin{enumerate}[label={(\alph*)}]
\item
The  condition \( p \conc \seq{ \set{ ( k , ( \dot{R} , \emptyset ) ) } }\) in the statement of \( ( \dagger )_n \) belongs to \( \forcing{P}_{ n + 1 } \) and it is obtained by extending \( p \) with the function with domain \( \setLR{k} \subseteq \omega \) such that \( k \mapsto ( \dot{R}, \emptyset ) \in \mathcal{R}_n \times {}^{ < \omega } \omega \).
It forces  the limit points of \( \Cl_A ( \ran ( f_{ n , k }) ) \) to belong to the finite union \( A_{H_0} \cup \bigcup \setof{ \ran ( f_{ i , j } ) }{ i < n  \text{ and } j \in H_{ i + 1 }} \).

\item
Note that \( p \conc \seq{ \set{ ( k , ( \dot{R} , \emptyset ) ) } }\) is the \( \le \)-maximum among the conditions \(q \le p \) such that \( q_{ n + 1 }^R ( k ) = \dot{R} \). 
Therefore, for any fixed \( n , k \), the set of conditions \( p \conc \seq{ \set{ ( k , ( \dot{R}, \emptyset ) ) } } \) we are considering in \( ( \dagger )_n \) is pre-dense in \( \forcing{P}_{ n + 1 } \) (and also in \(\forcing{P}\)). 
\end{enumerate}
\end{remarks*}

\begin{claim}\label{claim0}
Assume \( ( \dagger ) \). 
For each \( n , k \in \omega\),  \( \forces \bigl ( \Cl_{\dot{A}} ( \ran ( \dot{f}_{ n , k } ) ) \bigr )^{( n + 2 ) } = \emptyset \).
%\added{, that is to say: in \( \Vv [ G ] \) the set \( \Cl_A ( \ran f_{ n , k } ) \) is scattered, of height \(  \leq n + 2 \).}\comment[id=1]{Non vedo bene l'utilità di questa aggiunta in questo punto. L'ho spostata all'inizio della dimostrazione. }
\end{claim}

\begin{proof}
We prove the claim by induction on \( n \). 
Let \( n = 0 \) and fix \( k \in \omega \), \( p \in \forcing{P}_0 \), \( \dot{R} \in \mathcal{R}_0 \), \( \vec{H} = \seq{ H_0 } \) satisfying the hypotheses of  \( ( \dagger )_0 \). 
By \( ( \dagger )_0 \),
\[
 p \conc \seq{ \set{ ( k , ( \dot{R}, \emptyset ) ) } } \forces ( \Cl_{\dot{A}} ( \ran ( \dot{f}_{ n , k } ) ) )^{ ( 1 ) } \subseteq \dot{A}_{H_0}.
\]
Since \( H_0 \) is finite, we have that \( \forces ``\dot{A}_{H_0} \) is finite'', so our condition forces that \(  \Cl_A ( \ran f_{ n , k } ) \) has scattered height \(  \leq 2 \), that is
\[
p \conc \seq{ \set{ ( k , ( \dot{R}, \emptyset ) ) } } \forces \bigl ( \Cl_{\dot{A}} ( \ran ( \dot{f}_{ n , k } ) ) \bigr ) ^{( 2 ) } = \emptyset  .
\]
By density, the base case follows.

Now the induction step. 
Let \( n > 0 \) and fix \( k \in \omega \), \( p \in \forcing{P}_n \), \( \dot{R} \in \mathcal{R}_n \), \( \vec{H} = \seq{ H_0, \dots, H_n } \) satisfying the hypotheses of  \( ( \dagger )_n \). 
By \( ( \dagger )_n \),
\[
p \conc \seq{ \set{ ( k , ( \dot{R}, \emptyset ) ) } } \forces \bigl ( \Cl_{\dot{A}} ( \ran ( \dot{f}_{ n , k } ) ) \bigr )^{ ( 1) } \subseteq \dot{A}_{H_0} \cup \bigcup_{\smash[b]{\substack{i < n \\ j \in H_{ i + 1 } } }} \ran ( \dot{f}_{ i , j } ). 
\]
As the \( H_i \)s are all finite,
\[
p \conc \seq{ \set{ ( k , ( \dot{R}, \emptyset ) ) } } \forces \bigl ( \Cl_{\dot{A}} ( \ran ( \dot{f}_{ n , k } ) ) \bigr )^{ ( n + 2 ) } \subseteq \bigcup_{\smash[b]{\substack{i < n \\ j \in H_{ i + 1 }}}} \bigl ( \Cl_{\dot{A}} ( \ran ( \dot{f}_{ i , j } ) ) \bigr )^{ ( n + 1 ) } .
\]
By induction hypothesis, for all \( i < n \) and all \(  j \in H_{ i + 1 } \)
\[
%\FORALL{ i < n } \FORALL{ j \in H_{ i + 1 } } \left ( 
\forces \bigl (  \Cl_{\dot{A}} ( \ran ( \dot{f}_{ i , j } ) ) \bigr )^{ ( n + 1 ) } = \emptyset  ,
%\right ) ,
\]
and hence,
\[
p \conc \seq{ \set{ ( k , ( \dot{R}, \emptyset ) ) } } \forces \bigl ( \Cl_{\dot{A}} ( \ran ( \dot{f}_{ n , k } ) ) \bigr )^{ ( n + 2 ) } = \emptyset  .
\]
By density, the induction step follows.
\end{proof}

The statement \( ( \dagger ) \) is proved by induction on \( n \in \omega  \).

Let \( n = 0 \) and fix \( k \in \omega \), \( p \in \forcing{P}_0 \), \( \dot{R} \in \mathcal{R}_0 \), \( \vec{H} = \seq{ H_0 } \) satisfying the hypotheses of \( ( \dagger )_0 \).

\begin{claim}\label{claim1}
 \( p \conc \seq{ \set{ ( k , ( \dot{R}, \emptyset ) ) } } \forces \ran ( \dot{f}_{ 0 , k } ) \setminus \dot{A}_{ H_0 } \text{ is discrete.}\)
\end{claim}
 
\begin{proof}
Assume for a contradiction that there are \( q \leq p \conc \seq{ \set{ ( k , ( \dot{R}, \emptyset ) ) } } \) and \( l \in \omega \) such that 
\[
q \forces \dot{f}_{ 0 , k } ( l ) \notin \dot{A}_{H_0} \text{ and } \dot{f}_{ 0 , k } ( l ) \text{ is a limit point of } \ran ( \dot{f}_{ 0 , k } ) \setminus \dot{A}_{H_0}.
\]
Without loss of generality suppose that \( \lh ( q_1^s ( k ) ) > l + 1 \) and let \( m = q_1^s ( k ) ( l ) \), \( t = q_0 ( m ) \)---in particular, \( m \notin H_0 \) and \( q \forces \dot{f}_{ 0 , k } ( l ) = \dot{a}_m \in \dot{\Nbhd}_t \).
From our assumption and from the fact that \( H_0 \) is a finite set, it follows that there must be a \( z \leq q \) and an \( h > l \) such that 
\[
z \forces \dot{f}_{ 0 , k } ( h ) \in \dot{\Nbhd}_t\setminus \dot{A}_{ H_0 } .
\]
Assume without loss of generality \( \lh ( z_1^s ( k ) ) > h \) and let \( m' = z_1^s ( k ) ( h ) \), \( t' = z_0 ( m ' ) \)---in particular, \( m' \notin H_0 \), \( t' \supseteq t \) and \( z_0 \forces \dot{a}_m \mathrel{\dot{R}^+} \dot{a}_{m'} \). 
 Note that \( m' \neq m \), as otherwise \( z_0 \) would force \( \dot{R} \) to have a cycle, which is a contradiction, as \( z_0 \) extends \( p \) and by hypothesis \( p \) forces \( \dot{R} \) to be acyclic.
\begin{subclaim}\label{subclaim1}
\( p' = z_0 \restriction ( \omega \setminus \set{ m } ) \cup \set{ ( m , t' ) } \forces \dot{a}_{m} \mathrel{\dot{R}^+} \dot{a}_{m'}. \)
\end{subclaim}
\noindent A quick observation: since \( z_0 \le q_0 \), \( z_0 ( m ) \) surely extends \( t = q_0 ( m ) \), but a priori \( z_0 ( m ) \) could be incompatible with \( t' = z_0 ( m' ) = p' ( m ) \), making \( p' \) incompatible with \( z_0 \). 
Thus, our subclaim needs some care. 

\begin{proof}[Proof of the Subclaim.]
Let \( m_0 , m_1 , \dots , m_{ h - l } \in \omega \) be such that \( m_i = z_1^s ( k ) ( l + i ) \) for all \( i \le h - l \). 
Note that \( m_0 = m = q_1^s ( k ) ( l ) \) and \( m_1 = q_1^s ( k ) ( l + 1 ) \), since \( z \le q \). 
Moreover, \( m_{ h - l } = m' \), by definition of \( m' \). 
We can assume that the \( m_i \)s are all distinct, as otherwise \( z_0 \) would force \( \dot{R} \) to have a cycle.

Clearly \( q_0 \forces \dot{a}_{m_0} \mathrel{\dot{R}} \dot{a}_{m_1} \). 
By the Restriction Lemma, \( q_0 \restriction H_0 \cup \set{ m_0 , m_1 } \) forces the same.

On the other hand, \( z_0 \forces \dot{a}_{m_1} \mathrel{\dot{R}} \dot{a}_{m_{2}} \mathrel{\dot{R}} \dots \mathrel{\dot{R}} \dot{a}_{m_{ h - l }} \). 
Again by the Restriction Lemma, \( z_0 \restriction H_0 \cup \set{ m_1 , m_2 , \dots, m_{ h - l } } \) forces the same.

The condition \( p' = z_0 \restriction ( \omega \setminus \set{ m_0 } ) \cup \set{ ( m_0 , t' ) } \) extends both \( q_0 \restriction H_0 \cup \set{ m_0 , m_1 } \) and \( z_0 \restriction H_0 \cup \set{ m_1 , m_2 , \dots , m_{ h - l } } \)---here use the fact that \( m_0 \notin H_0 \) and that all the \( m_i \)s are distinct. 
Hence \( p' \) forces \( \dot{a}_{m_0} \mathrel{\dot{R}} \dot{a}_{m_{1}} \mathrel{\dot{R}} \dot{a}_{m_{2}} \mathrel{\dot{R}} \dots \mathrel{\dot{R}} \dot{a}_{m_{h-l}} \).
\end{proof}

Let \( \pi_0 \colon \omega \to \omega \) be the permutation that swaps \( m \) and \( m' \) fixing everything else---in particular, \( \pi_0 \in \Fix (H_0) \) and \( \pi_0 \dot{R} = \dot{R} \). 
Then, by the Symmetry Lemma,
\[
\pi_0 p' = p' \forces \dot{a}_{m'} \mathrel{\dot{R}^+} \dot{a}_m,
\] 
but then \( p' \) both extends \( p \) and forces \( \dot{a}_{m} \mathrel{\dot{R}^+} \dot{a}_m \), which is a contradiction, since we assumed that \( p \) forces \( \dot{R} \) to be acyclic.
\end{proof}

By Claim~\ref{claim1}, condition \( p \conc \seq{ \set{ ( k , ( \dot{R}, \emptyset ) ) } } \) forces that the limit points of \( \ran ( f_{ 0 , k } ) \) belong to the finite set \( A_{H_0 } \).
The next claim shows that the same is true for the larger set \( \Cl_{A} (  \ran ( f_{ 0 , k } ) ) \).

\begin{claim}\label{claim2}
\( p \conc \seq{ \set{ ( k , ( \dot{R}, \emptyset ) ) } } \forces \bigl ( \Cl_{\dot{A}} ( \ran ( \dot{f}_{ 0 , k } ) ) \bigr )^{ ( 1 ) } \subseteq \dot{A}_{ H_0 } \).
\end{claim}

\begin{proof}
Suppose for a contradiction that the claim is false. Then there is a \( q \leq p \conc \seq{ \set{ ( k , ( \dot{R}, \emptyset ) ) } } \) and an \( m \notin H_0 \) such that 
\begin{equation}\label{eq:claim2}
q \forces \dot{a}_m \text{ is a limit point of }\ran ( \dot{f}_{ 0 , k } ) .
\end{equation}
Note that, since \( H_0 \) is finite, \( q \) actually forces \( \dot{a}_m \) to be a limit point of \( \ran ( \dot{f}_{ 0 , k } ) \setminus \dot{A}_{ H_0 } \). 
Hence, it follows from Claim~\ref{claim1} that \( q \) forces \( \dot{a}_m \) not to be in the range of \( \dot{f}_{ 0 , k } \). 
In particular, \( m \notin \ran ( q_1^s ( k ) ) \).

The condition \( q' = \seq{ q_0, q_1 \restriction \set{ k } } \) extends \( p \) and, by the Restriction Lemma, still forces \eqref{eq:claim2}.
Let \( t \) be \( q_0 ( m ) \)---in particular, \( q' \forces \dot{a}_m \in \dot{\Nbhd}_t \).
 
We now show \( q' \forces \dot{\Nbhd}_t \subseteq \Cl ( \ran ( \dot{f}_{ 0 , k } ) \setminus \dot{A}_{H_0} ) \), which clearly contradicts Claim~\ref{claim1}, as every discrete set of reals is nowhere dense. 
Pick any \( z \leq q' \) and a \( t'\supseteq t \). 
Fix an \( m' \notin H_0 \cup \dom ( z_0 ) \cup \ran ( q_1^s ( k ) ) \). 
Define \( z' \) to be the condition such that \( z_0' = z_0 \cup \set{ ( m' , t' ) } \) and \( z_i' = z_i \) for every \( i > 0 \). 

Now, \( z' \) clearly extends \( z \). 
Moreover, if we let \( \pi_0 \) be the permutation of \( \omega \) that swaps \( m \) and \( m' \), \( z' \) also extends \( \seq{ \pi_0 } q' \). 
Indeed, since \( t' \supseteq t \), it's clear that \( z'_0 \) extends \( \pi_0 q_0 \). 
But since both \( m \) and \( m' \) do not belong to \( H_0 \cup \ran ( q_1^s ( k ) ) \), we also have \( ( \seq{ \pi_0 } q')_1 = q'_1 \), and therefore \( \seq{ z'_0, z'_1} = \seq{ z'_0, z_1 } \) extends \( \seq{ \pi_0 } q' = \seq{ \pi_0 q_0 , q'_1 } \). 
Overall, \( z' \) extends \( \seq{ \pi_0 } q' \).

By \eqref{eq:claim2} and the Symmetry Lemma,
\[ 
\seq{ \pi_0 } q' \forces \dot{a}_{m'} \in \Cl ( \ran ( \dot{f}_{ 0 , k } ) \setminus \dot{A}_{H_0} ) .
\] 
Since \( z' \) extends \( \seq{ \pi_0 } q' \) and \( z' \forces \dot{a}_{m'} \in \dot{\Nbhd}_{t'} \), we have
\[ 
z' \forces \dot{\Nbhd}_{t'} \cap \Cl ( \ran ( \dot{f}_{ 0 , k } ) \setminus \dot{A}_{H_0}) \neq \emptyset.
\]
By density,
\[
q' \forces \dot{\Nbhd}_t \subseteq \Cl ( \ran ( \dot{f}_{ 0 , k } ) \setminus \dot{A}_{H_0} ).\qedhere
\]
\end{proof}
This proves \( ( \dagger )_0 \).

Here comes the induction step: fix an \( n > 0 \) and suppose \( ( \dagger )_i \) holds for every \( i < n \), towards proving \( ( \dagger )_n \).
Fix \( k \in \omega \), \( p \in \forcing{P}_n \), \( \dot{R} \in \mathcal{R}_n \), \( \vec{H} = \seq{ H_0, \dots, H_n } \) satisfying the hypotheses of \( ( \dagger )_n \).
The next claim is the analogue of Claim~\ref{claim1}.

\begin{claim}\label{claim3}
\[
p \conc \seq{ \set{ ( k , ( \dot{R}, \emptyset ) ) } } \forces \ran ( \dot{f}_{ n , k }) \setminus \Bigl ( \dot{A}_{H_0} \cup \bigcup_{\smash[b]{\substack{i < n \\ j \in H_{ i + 1 }}}}  \ran ( \dot{f}_{ i , j })  \Bigr ) \text{ is discrete}.
\]
\end{claim}

\begin{proof}
Suppose for a contradiction that there are \( q \leq \seq{ p, \set{ ( k , ( \dot{R} , \emptyset ) ) } } \) and \( l \in \omega \) such that 
\begin{multline*}
q \forces \dot{f}_{ n , k } ( l ) \notin \dot{A}_{H_0} \cup \smash[b]{\bigcup_{\substack{i < n \\ j \in H_{ i + 1 } } }} \ran ( \dot{f}_{ i , j }) \text{ and }\dot{f}_{ n , k } ( l ) \text{ is a limit point of }
\\
\ran ( \dot{f}_{ n , k } ) \setminus \Bigl ( \dot{A}_{H_0} \cup \bigcup_{\smash[b]{\substack{i < n \\ j \in H_{ i + 1 } }}} \ran ( \dot{f}_{ i , j } ) \Bigr ) .
\end{multline*}
Suppose without loss of generality that \( \lh ( q_{ n + 1 } ^s ( k ) ) > l + 1 \) and let \( m = q_{ n + 1 }^s ( k ) ( l ) \), and \( t = q_0 ( m ) \)---in particular, \( q \forces \dot{f}_{ n , k } ( l ) = \dot{a}_m \in \dot{\Nbhd}_t \). 
By assumption there must be a \( z \leq q \) and an \( h > l \) such that
\[
z \forces \dot{f}_{ n , k } ( h ) \in \dot{\Nbhd}_t \setminus \bigg ( \dot{A}_{ H_0 } \cup \bigcup_{\smash[b]{ \substack{i < n \\ j \in H_{i + 1}}}} \ran ( \dot{f}_{ i , j } ) \bigg). 
\]
Assume without loss of generality \( \lh ( z_{n + 1 }^s ( k ) ) > h \) and let \( m' = z_{ n + 1 }^s ( k ) ( h ) \), and \( t' = z_0 ( m ' ) \)---in particular \( t' \supseteq t \) and \( z \restriction n + 1 \forces \dot{a}_m \mathrel{\dot{R}^+ } \dot{a}_{m'} \). 
Since 
\[
z \forces \dot{a}_m, \dot{a}_m' \notin \bigg ( \dot{A}_{H_0} \cup \bigcup_{\smash[b]{ \substack{i < n \\ j \in H_{i + 1}}}} \ran ( \dot{f}_{ i , j } ) \bigg),
\]
then, in particular,
\begin{equation}\label{eq:claim3}
m,m' \notin H_0 \cup \bigcup_{\smash[b]{\substack{i < n \\ j \in H_{i + 1}}} } \ran \big(z^s_{i+1}(j)\big).
\end{equation}
Now let
\[
p' = \seq{ z_0 \restriction ( \omega \setminus \set{ m } ) \cup \set{ ( m , t' ) } , z_1 \restriction H_1, \dots, z_n \restriction H_n } .
\] 
By an argument analogous to the one used in the proof of Subclaim~\ref{subclaim1}, we can show that
\[
p' \forces \dot{a}_m \mathrel{ \dot{R}^+ } \dot{a}_{m'}.
\]
If we let \( \pi_0 \colon \omega \to \omega \) be the permutation that swaps \( m \) and \( m' \), then \( \seq{ \pi_0 } p' = p' \). 
Indeed, it directly follows from the definition of \( p' \) that \( \pi_0 p'_0 = p'_0 \). 
Moreover, by \eqref{eq:claim3}, both \( m \) and \( m' \) do not belong to \( H_0 \), hence \( \seq{ \pi_0 } \in \Fix ( \vec{H} ) \). 
As such, \( ( \seq{ \pi_0 } p' )_i^R = ( p' )_i^R \) for every \( 1 \le i \le n \). 
Again by~\eqref{eq:claim3}, \( m \) and \( m' \) do not belong to the range of \( ( p' )_i ^s ( j ) \) for any \( 1 \le i \le n \) and \( j \in \dom ( p'_i ) = H_i \), and therefore \( ( \seq{ \pi_0 } p' )_i ^s = ( p' )_i ^s \) for every \( 1 \le i \le n \). 
Overall, \( \seq{ \pi_0 } p' = p' \).

Next note that \( \seq{ \pi_0 } \dot{R} = \dot{R} \), as \( \seq{ \pi_0 } \in \Fix ( \vec{H} ) \). 
By the Symmetry Lemma,
\[
\seq{ \pi_0 } p' = p' \forces \dot{a}_{m'} \mathrel{ \dot{R}^+ } \dot{a}_{m} ,
\] 
but then \( p' \) both extends \( p \) and forces \( \dot{a}_m \mathrel{ \dot{R}^+ } \dot{a}_m \), which is a contradiction, since we assumed that \( p \) forces \( \dot{R} \) to be acyclic.
\end{proof}

\begin{claim}\label{claim4}
\[
 p\forces \dot{A}_{H_0} \cup \bigcup_{\smash[b]{\substack{i < n \\ j \in H_{i+1}}}} \ran ( \dot{f}_{i,j}) \text{ is closed with respect to } \dot{A}. 
\]
\end{claim}
\begin{proof}
Fix \( q \le p \) and \( m \) such that 
\[
q \forces \dot{a}_m \in \Cl_{\dot{A}}\Bigl ( \dot{A}_{H_0} \cup \bigcup_{ \smash[b]{\substack{i < n \\ j \in H_{ i + 1 } } } } \ran ( \dot{f}_{ i , j } ) \Bigr ) .
\]
We would like to prove that there is a condition \( z \le q \) such that 
\[
z \forces \dot{a}_m \in \dot{A}_{H_0} \cup \bigcup_{\smash[b]{ \substack{i < n \\ j \in H_{ i + 1 }}} } \ran ( \dot{f}_{ i , j } ) ,
\]
so, to avoid trivialities, we assume
\[
q \forces \dot{a}_m \in \Bigl ( \Cl_{\dot{A}} \bigl ( \dot{A}_{H_0} \cup \bigcup_{\smash[b]{ \substack{i < n \\ j \in H_{ i + 1 } } } } \ran ( \dot{f}_{ i , j } ) \bigr ) \Bigr )^{ ( 1 ) }.
\]

As the \( H_i \)s are finite, there exists a \( z \le q \), an \( i < n \) and some \( j \in H_{ i + 1 } \) such that \( z \forces \dot{a}_m \in ( \Cl_{\dot{A}} ( \ran ( \dot{f}_{ i , j } ) ) )^{ ( 1 ) } \). 
But then, by \( ( \dagger )_i \) (here we use our induction hypothesis), 
\[
z \forces \dot{a}_m \in \bigl ( \Cl_{\dot{A}} ( \ran ( \dot{f}_{ i , j } ) ) \bigr )^{ ( 1 ) } \subseteq \dot{A}_{ H_0 } \cup \bigcup_{\smash[b]{\substack{l < i \\ h \in H_{ l + 1 } } } } \ran ( \dot{f}_{ l , h } ) .
\]
By density, the claim follows.
\end{proof}

The next claim is the analogue of Claim~\ref{claim2}.

\begin{claim}\label{claim5}
\[
p \conc \seq{ \set{ ( k , ( \dot{R}, \emptyset ) ) } } \forces \bigl ( \Cl_{\dot{A}} ( \ran ( \dot{f}_{ n , k } ) ) \bigr )^{ ( 1 ) } \subseteq \dot{A}_{H_0} \cup \bigcup_{\smash[b]{ \substack{i < n \\ j \in H_{ i + 1 } } } } \ran ( \dot{f}_{ i , j } ) . 
\]
\end{claim}

\begin{proof}
Suppose for a contradiction that this is not the case, then there is a \( q \leq p \conc \seq{ \set{ ( k , ( \dot{R}, \emptyset ) ) } } \) and an \( m \) such that
\[
q \forces \dot{a}_m \text{ is a limit point of } \ran ( \dot{f}_{ n , k } ) \text{ and } \dot{a}_m \notin \dot{A}_{ H_0 } \cup \bigcup_{\smash[b]{ \substack{i < n \\ j \in H_{ i + 1 }}}} \ran ( \dot{f}_{ i , j } ) .
\]
From Claim~\ref{claim4} it follows that 
\begin{equation}\label{eq:claim5}
q \forces \dot{a}_m \text{ is a limit point of } \ran ( \dot{f}_{ n , k } ) \setminus \Bigl ( \dot{A}_{ H_0 } \cup \bigcup_{\smash[b]{\substack{i < n \\ j \in H_{ i + 1 } } } } \ran ( \dot{f}_{ i , j } ) \Bigr ) .
\end{equation}

But then, by Claim~\ref{claim3}, \( q \) also forces \( \dot{a}_m \) not to be in the range of \( \dot{f}_{ n , k } \). 
In particular,
\begin{equation}\label{eq:claim5m1}
m \notin H_0 \cup \ran ( q_{ n + 1 }^s ( k ) ) \cup \bigcup_{\smash[b]{\substack{i < n \\ j \in H_{ i + 1 } } } } \ran ( q_{ i + 1 }^s ( j ) ).
\end{equation}

Let 
\[
q' = \seq{ q_0, q_1 \restriction H_1, \dots, q_n \restriction H_n, q_{ n + 1 } \restriction \set{ k } } .
\] 
Then \( q' \) extends \( p \) and, by the Restriction Lemma, still forces \eqref{eq:claim5}.
Let \( t \) be \( q_0 ( m ) \)---in particular \( q' \forces \dot{a}_m \in \dot{\Nbhd}_t \). 
 
We now show that
\[
q' \forces \dot{\Nbhd}_t \subseteq \Cl \Bigl ( \ran ( \dot{f}_{ n , k } ) \setminus \bigl ( \dot{A}_{H_0} \cup \bigcup_{\substack{i < n \\ j \in H_{ i + 1 } } } \ran ( \dot{f}_{ i , j } ) \bigr ) \Bigr ) ,
\]
which contradicts Claim~\ref{claim3}. 
Pick any \( z \leq q' \) and \( t'\supseteq t \). 
Fix an \( m' \in \omega \) such that 
\begin{equation}\label{eq:claim5m2}
m' \notin H_0 \cup \dom ( z_0 ) \cup \ran ( q_{ n + 1 }^s ( k ) ) \cup \bigcup_{\substack{i < n \\ j \in H_{i + 1 } } } \ran ( q_{ i + 1 }^s ( j ) ) .
\end{equation}

Define \( z' \) to be the condition such that \( z'_0 = z_0 \cup \set{ ( m' , t' ) } \) and \( z_i' = z_i \) for all \( i > 0 \). 
Now, \( z' \) clearly extends \( z \). 
Moreover, if we let \( \pi_0 \) be the permutation of \( \omega \) that swaps \( m \) and \( m' \), \( z' \) also extends \( \seq{ \pi_0 } q' \). 
Indeed, since \( t' \supseteq t \), it's clear that \( z'_0 \) extends \( \pi_0 q_0 \). 
But from~\eqref{eq:claim5m1} and~\eqref{eq:claim5m2}, it follows that \( ( \seq{ \pi_0 } q' )_i = q'_i \) for every \( 1 \le i \le n+1 \), and therefore \( z' \) extends \( \seq{ \pi_0 } q' \).

By~\eqref{eq:claim5} and the Symmetry Lemma,
\[
\seq{ \pi_0 } q' \forces \dot{a}_{m'} \in \Cl \Bigl ( \ran ( \dot{f}_{ n , k } ) \setminus \bigl ( \dot{A}_{H_0} \cup \bigcup_{\smash[b]{\substack{i < n \\ j \in H_{ i + 1 } } } } \ran ( \dot{f}_{ i , j } ) \bigr ) \Bigr ) .
\]
Since \( z' \) extends \( \seq{ \pi_0 } q' \) and \( z' \forces \dot{a}_{m'} \in \dot{\Nbhd}_{t'} \), we have
\[
z' \forces \dot{\Nbhd}_{t'} \cap \Cl \Bigl ( \ran ( \dot{f}_{ n , k } ) \setminus \bigl ( \dot{A}_{ H_0 } \cup \bigcup_{\smash[b]{\substack{i < n \\ j \in H_{ i + 1 } } } } \ran ( \dot{f}_{ i , j } ) \bigr ) \Bigr ) \neq \emptyset.
\]
By density, 
\[
q' \forces \dot{\Nbhd}_t \subseteq \Cl \Bigl ( \ran ( \dot{f}_{ n , k } ) \setminus \bigl ( \dot{A}_{ H_0 } \cup \bigcup_{\substack{i < n \\ j \in H_{ i + 1 } } } \ran ( \dot{f}_{ i , j } ) \bigr ) \Bigr ) . \qedhere
\]
\end{proof}
This completes the proof of  \( ( \dagger )_n \), so by induction \( ( \dagger ) \) holds.
By Claim~\ref{claim0}, we are done.
\end{proof}

In light of Proposition~\ref{lem:core}, we can prove that in \( \mathcal{N} \) every separable subset of \( A \) is scattered with finite scattered height.

\begin{theorem}\label{th:separablesub}
In the model \( \mathcal{N} \) the following holds: for every separable \( S \subseteq A \) there is an \( n \in \omega \) such that \( S^{ ( n ) } = \emptyset \).
\end{theorem}

\begin{proof}
Let \( S \in \mathcal{N} \) be a separable subset of \( A \) and fix in \( \mathcal{N} \) a function \( f \colon \omega \to A \) such that \( S \subseteq \Cl ( \ran ( f ) ) \). 
Then there must be a \( p \in G \) such that 
\[
p \forces \dot{f} \colon \check{\omega} \to \dot{A},
\]
where \( \dot{f} \in \HS \) is a symmetric name for \( f \), with support \( \vec{H} = \seq{ H_0, \dots, H_n } \).
We can assume without loss of generality that \( \dom (p_i) = H_i \) for each \( i \), and that for all \( i > 0 \), for all \( j \in H_i \), \( \vec{H}\restriction i \) is a support for \( p_{i}^R ( j ) \).
We claim that 
\[
p \forces \ran ( \dot{f} ) \subseteq \dot{A}_{H_0} \cup \bigcup_{\substack{i < n \\ j \in H_{ i + 1 } } } \ran ( \dot{f}_{ i , j } ) ,
\]
where \( \dot{A}_{H_0} \) is the \( \forcing{P} \)-name as in~\eqref{eq:dotA_H}.
If we manage to do so, then Proposition~\ref{lem:core} ensures that \( \Cl_A ( \ran ( f ) ) \) is scattered of height \( \leq n + 2 \), and, \emph{a fortiori}, that \( S^{( n + 2 )}  = \emptyset \), as required.
 
Suppose that the claim is false, then there exist \( q \leq p\) and \( l , m \in \omega \) such that 
\begin{equation}\label{eq:thmsepsub1}
q \forces \dot{f} ( l ) = \dot{a}_m \notin \dot{A}_{H_0} \cup\bigcup_{\smash[b]{\substack{i < n \\ j \in H_{i + 1 } } }} \ran ( \dot{f}_{ i , j } ) .
\end{equation}
In particular, 
\begin{equation}\label{eq:thmsepsub2}
m \notin H_0 \cup \bigcup_{\substack{i < n \\ j \in H_{i + 1 } } } \ran ( q_{ i + 1 } ^s ( j ) ) .
\end{equation}

Let \( q' = \seq{ q_0, q_1 \restriction H_1, \dots, q_n \restriction H_n } \). 
Then, by the Restriction Lemma, \( q' \) still forces~\eqref{eq:thmsepsub1}.

Fix an \( m' \in \omega \) such that 
\begin{equation}\label{eq:thmsepsub3}
m' \notin H_0 \cup \dom ( q_0 ) \cup \bigcup_{\substack{i < n \\ j \in H_{i + 1 } } } \ran (q_{ i + 1 } ^s ( j ) ) .
\end{equation}
Let \( \pi_0 \) be the permutation of \( \omega \) that swaps \( m \) and \( m' \), then \( \seq{ \pi_0 } q' \) and \( q' \) are compatible. 
Indeed, since \( m' \notin \dom (q_0') \), then \( q'_0 \) and \( \pi_0 q'_0 \) are clearly compatible. 
Moreover, it follows from \eqref{eq:thmsepsub2} and \eqref{eq:thmsepsub3} that \( ( \seq{ \pi_0 } q')_i = q'_i \) for every \( 1 \le i \le n \), and therefore \( \seq{ \pi_0 } q' \) and \( q' \) are compatible. 
By the Symmetry Lemma,
\[
\seq{ \pi_0 } q' \forces \dot{f} ( l ) = \dot{a}_{m'} .
\]
So \( q' \) and \( \seq{ \pi_0 } q' \), while being compatible, force \( \dot{f} \) to take different values at \( l \), but they both extend \( p \), which forces \( \dot{f} \) to be a function. 
Contradiction.
\end{proof}

\begin{corollary}\label{cor:separable}
\( \mathcal{N} \Models \neg \AComega ( A ) \).
\end{corollary}

\begin{proof}
Assume for a contradiction that \( \AComega ( A ) \) holds, then \( A \) is certainly separable. 
By Theorem~\ref{th:separablesub}, \( A \) would be scattered. 
But \( A \) has no isolated points. 
Contradiction.
\end{proof}

Now we are left to prove that \( \DC ( A ) \) holds in \( \mathcal{N} \). 
Let \( \dot{\mathcal{N}}_n \) be the canonical name for the intermediate model \( \mathcal{N}_n \).

\begin{lemma}\label{lem:modcl}
Let \( n \in \omega \) and \( \dot{x} \in \HS \) with support \( \vec{H} = \seq{ H_0, \dots, H_n } \), then 
\[
 \forces \dot{x} \subseteq \dot{\mathcal{N}}_n \IMPLIES \dot{x} \in\dot{\mathcal{N}}_n \, .
\]
\end{lemma}

\begin{proof}
Fix \( ( \dot{y} , p ) \in \dot{x} \).
As the set of \( q \) such that either \( q \forces \dot{y} \in \dot{ \mathcal{N}}_n \) or else \( q \forces \dot{y} \notin \dot{ \mathcal{N}}_n \) is dense below \( p \), there is a maximal antichain \( D_{ ( \dot{y} , p ) } \) below \( p \) and a map \( h_{ ( \dot{y} , p ) } \colon D_{ ( \dot{y} , p ) } \to \HS _n \) such that, for each \( q \in D_{ ( \dot{y} , p ) } \), either \( q \forces \dot{y} = h_{ ( \dot{y} , p ) } ( q ) \) or \( q \forces \dot{y} \notin \dot{\mathcal{N}}_n \). 
Let \( D'_{ ( \dot{y} , p ) } = \setof{ q \in D_{ ( \dot{y} , p ) } }{ q \forces \dot{y} \in \dot{\mathcal{N}}_n } \) and let
\[
C = \setofLR{ \vec{\pi} \bigl ( h_{ ( \dot{y} , p ) } ( q ) \bigr ) }{ ( \dot{y} , p ) \in \dot{x}, \ q \in D'_{ ( \dot{y} , p ) } , \ \vec{\pi} \in \Fix ( \vec{H} ) } .
\] 
Consider the following \( \forcing{P} _n \)-name:
\[
\dot{w} = \setof{ ( \dot{y} , q ) }{ \dot{y} \in C, \ q \in \forcing{P} _n \text{ and } q \forces \dot{y} \in \dot{x} }.
\]

\begin{claim}
\( \dot{w} \in \HS _n \) with support \( \vec{H} \).
\end{claim}

\begin{proof}
Let \( \vec{\pi} \in \Fix ( \vec{H} ) \) and \( ( \dot{y} , q ) \in \dot{w} \). 
By definition, \( q \forces \dot{y} \in \dot{x} \), hence \( \vec{\pi} q \forces \vec{\pi} \dot{y} \in \dot{x} \). 
Since \( \vec{\pi} \dot{y} \in C \), this means that \( ( \vec{\pi} \dot{y}, \vec{\pi} q ) \in \dot{w} \). 
Hence \( \vec{\pi} \dot{w} = \dot{w} \).
\end{proof}

Fix \( p \in \forcing{P} \) such that \( p \forces \dot{x} \subseteq \dot{\mathcal{N}}_n \).

\begin{claim}
\( p \forces \dot{w} = \dot{x} \).
\end{claim}

\begin{proof}
Let \( q \leq p \) and \( \dot{z} \in \HS \) such that \( q \forces \dot{z} \in \dot{x} \). 
By definition of \( C \) and our hypothesis on \( p \), there is an \( r \leq q \) and a \( \dot{y} \in C \) such that \( r \forces \dot{z} = \dot{y} \in \dot{x} \). 
By the Restriction Lemma, \( r \restriction n + 1 \forces \dot{y} \in \dot{x} \), hence \( ( \dot{y} , r \restriction n + 1 ) \in \dot{w} \) and, in particular, \( r \forces \dot{z} = \dot{y} \in \dot{w} \). 
By density, \( p \forces \dot{x} \subseteq \dot{w} \).

The other inclusion is immediate from the definition of \( \dot{w} \).
\end{proof}
Therefore \( p \forces \dot{x} \in \dot{\mathcal{N}}_n \).
By density, \( \forces \dot{x} \subseteq \dot{\mathcal{N}}_n \IMPLIES \dot{x} \in \dot{ \mathcal{N} }_n \).
\end{proof}

\begin{theorem}\label{th:dc}
\( \mathcal{N} \Models \DC ( A ) \).
\end{theorem}
\begin{proof}
 Since every binary relation \( R \in \mathcal{N} \) on \( A \) is a subset of \( A\times A \in \mathcal{N}_0 \), it follows from Lemma~\ref{lem:modcl} that \( R \in \mathcal{N}_n \) for some \( n \). 
 Now, either \( R \) is cyclic, but then it surely has a chain, or it is acyclic, but then Lemma~\ref{lem:newchains} says that in \( \mathcal{N}_{n+1} \subseteq \mathcal{N} \) there is a chain for this relation.
\end{proof}
This finishes the proof of Theorem \ref{th:Lorenzo}.

\section{Some complementary results}\label{sec:complementary}
We collect some facts related to our main results, and conclude with some open questions.
\subsection{Dependent Choice propagates under finite unions.}
By Proposition~\ref{prop:basicpropertiesofDC&AComega}, the axiom \( \DC (X) \) is closed under surjective images and, hence, under subsets.
The next result shows that it is also closed under finite unions.

\begin{theorem}\label{th:DCunion}
\( \DC ( X ) \wedge \DC ( Y ) \IMPLIES \DC ( X \cup Y ) \).
\end{theorem}

\begin{corollary}\label{cor:DCunion}
\( \DC ( X ) \IMPLIES \DC ( X \times n ) \), for all sets \( X \) and all \( n \in \omega \).
\end{corollary}

The natural progression from Corollary~\ref{cor:DCunion} would be to prove that \( \DC ( X ) \IMPLIES \DC ( X \times \omega ) \), but this cannot be established in \( \ZF \), since \( \DC ( X \times \omega ) \) implies \( \AComega ( X ) \) (part~\ref{prop:basicpropertiesofDC&AComega-e} of Proposition~\ref{prop:basicpropertiesofDC&AComega}) and we know from Theorem~\ref{th:Lorenzo} that \( \DC ( X ) \) does not necessarily imply \( \AComega ( X ) \).

If a binary relation \( R \) is such that \( \ran ( R ) \subseteq \dom ( R ) \), then it is total on its domain.
The largest \( R' \subseteq R \) such that \( \ran ( R' ) \subseteq \dom ( R' ) \) is
\[
\mathcal{D} ( R ) = \bigcup \setofLR{S \subseteq R}{ \ran ( S ) \subseteq \dom ( S )} .
\] 
By part~\ref{prop:basicpropertiesofDC&AComega-a} of Proposition~\ref{prop:basicpropertiesofDC&AComega} it is easy to see that
\begin{equation}\label{eq:DCdependable}
\DC ( X ) \IFF \FORALL{ R \subseteq X^2 } ( \mathcal{D} ( R ) \neq \emptyset \IMPLIES \text{there is a \( \mathcal{D} ( R ) \)-chain} ) . 
\end{equation}

\begin{proof}[Proof of Theorem~\ref{th:DCunion}]
Suppose \( \DC ( X ) \) and \( \DC ( Y ) \), and let \( R \subseteq ( X \cup Y )^2 \) be total, towards proving that there is an \( R \)-chain.
Without loss of generality, we may assume that \( X \) and \( Y \) are nonempty and disjoint.
If \( \mathcal{D} ( R \restriction X ) \neq \emptyset \), then by \( \DC ( X ) \) and ~\eqref{eq:DCdependable} there is a \( \mathcal{D} ( R \restriction X ) \)-chain, which is, in particular an \( R \)-chain.
Similarly, if \( \mathcal{D} ( R \restriction Y ) \neq \emptyset \), then there is an \( R \)-chain.
Therefore, without loss of generality, we may assume that \( R \) is acyclic, and that 
\begin{equation}\label{eq:th:DCunion0}
\mathcal{D} ( R \restriction X ) = \mathcal{D} ( R \restriction Y ) = \emptyset . 
\end{equation}

Recall that \( R^+ \) is the smallest transitive relation containing \( R \).
If \( x \in X \cup Y \) and \( R^+ ( x ) \subseteq X \), then \( R \restriction R^+ ( x ) \) would witness that \( \mathcal{D} ( R \restriction X ) \neq \emptyset \), against~\eqref{eq:th:DCunion0}. Similarly \( R^+ ( x ) \) cannot be included in \( Y \).
Therefore 
\begin{equation}\label{eq:th:DCunion1}
 \FORALL{x \in X \cup Y} ( R^+ ( x ) \nsubseteq X \wedge R^+ ( x ) \nsubseteq Y ).
\end{equation}

Here is the idea of the proof.
By~\eqref{eq:th:DCunion0}, any \( R \)-chain \( ( z_n )_{ n \in \omega } \) must visit both \( X \) and \( Y \) infinitely often, so \( ( z_n )_{ n \in \omega } \) can be seen as the careful merging of two sequences \( ( x_n )_{ n \in \omega } \) in \( X \) and \( ( y_n )_{ n \in \omega } \) in \( Y \).
The sequence \( ( x_n )_{ n \in \omega } \) is obtained by applying \( \DC ( X ) \) to a total relation \( R_X \) on \( X \) such that \( R \restriction X \subseteq R_X \subseteq R^+ \).
Using \( ( x_n )_{ n \in \omega } \), a suitable total relation \( R_Y \) on some \( Y' \subseteq Y \) is defined, and by \( \DC ( Y ) \) the required sequence \( ( y_n )_{ n \in \omega } \) is obtained.
Here come the details.

Let \( R_X \) be the relation on \( X \) given by \( R \restriction X \), together with all pairs \( ( x , x' ) \) such that \( x \mathrel{R} y_0 \mathrel{R} y_1 \mathrel{R} \cdots \mathrel{R} y_n \mathrel{R} x' \) for some finite sequence of elements of \( Y \):
\begin{multline*}
R_X = ( R \restriction X ) \cup \{ ( x , x' ) \in X^2 \Mid \EXISTS{m \geq 1 } \EXISTS{s \in {}^{m } Y}
\\
 ( x \mathrel{R} s ( 0 ) \wedge s ( m - 1 ) \mathrel{R} x' \wedge \FORALL{i < m - 1 } ( s ( i ) \mathrel{R} s ( i + 1 ) ) ) \} .
\end{multline*}
It is immediate that \( R_X \subseteq R^+ \).

\begin{claim}
\( R_X \) is total on \( X \).
\end{claim}

\begin{proof}
We must show that \( \dom ( R_X ) = X \).
Let \( x \in X \). 
If \( R ( x ) \cap X \neq \emptyset \), then \( x \in \dom ( R \restriction X) \subseteq \dom (R_X) \).

Now suppose otherwise. 
By~\eqref{eq:th:DCunion1} \( R^+ ( x ) \nsubseteq Y \), so there are \( y_0 , \dots , y_n \in Y \) and \( x' \in X \) such that \( x \mathrel{R} y_0 \mathrel{R} \dots \mathrel{R} y_n \mathrel{R} x' \).
Thus \( ( x , x' ) \in R_X \), so \( x \in \dom ( R_X ) \).
\end{proof}

By \( \DC ( X ) \) there is an \( R_X \)-chain \( ( x_n )_{ n \in \omega } \).

\begin{claim}
\( \FORALL{n} \EXISTS{ m > n } \neg ( x_{m} \mathrel{R} x_{m+1} ) \).
\end{claim}

\begin{proof}
Towards a contradiction, suppose that there is \( \bar{n}\in \omega \) such that \( x_{m} \mathrel{R} x_{ m + 1 } \) for every \( m \geq \bar{n} \).
Then \( R \restriction \setof{ x_m}{ m \geq \bar{n}} \) is total on \( \setof{ x_m}{ m \geq \bar{n}} \) and contained in \( R \restriction X \), against~\eqref{eq:th:DCunion0}.
\end{proof}

Let \( ( n_k )_{ k \in \omega } \) be the sequence enumerating the set of \( m \)s such that \( \neg ( x_{m} \mathrel{R} x_{ m + 1 } ) \).
By the definition of \( R_X \), each \( x_{n_k} \) is linked to \( x_{ n_k + 1} \) via \( R \) through some finite path in \( Y \), and let \( Y_k \) be the collection of all places visited by these paths:
\begin{multline*}
Y_k \coloneqq \bigcup \bigl \{ \ran ( s ) \Mid \EXISTS{m } \bigl ( s \in {}^{m+ 1} Y \wedge x_{n_k} \mathrel{R} s ( 0 ) \wedge s ( m ) \mathrel{R} x_{ n_k + 1 } \wedge{}
\\
 \FORALL{i < m } ( s ( i ) \mathrel{R} s ( i + 1 ) ) \bigr ) \bigr \} . 
\end{multline*}

\begin{claim}
The \( Y_k \)s are nonempty, pairwise disjoint subsets of \( Y \).
\end{claim}

\begin{proof}
For each \( k \) we have \( ( x_{ n_k } , x_{ n_k + 1 } ) \in R_X \setminus R \). 
This means that there is some \( \seq{ y_0 , \dots , y_m } \in {}^{<\omega} Y \) such that \( x_{n_k} \mathrel{R} y_0 \mathrel{R} \dots \mathrel{R} y_m \mathrel{R} x_{n_k + 1 } \). 
In particular, \( Y_k \neq \emptyset \).

Towards a contradiction, suppose there are indices \( k < j \) such that \( Y_k \cap Y_j \neq \emptyset \). 
Pick \( y \in Y_k \cap Y_j \).
Then \( y \mathrel{R^+} x_{ n_k + 1 } \mathrel{R^+} x_{n_j} \mathrel{R^+} y \), if \( x_{ n_k + 1 } \neq x_{n_j} \), or \( y \mathrel{R^+} x_{ n_k + 1 } = x_{n_j} \mathrel{R^+} y \) otherwise.
Either way, this contradicts our assumption that \( R \) is acyclic.
\end{proof}

Now we let \( R_Y \) be the relation on \( \bigcup_{k \in \omega} Y_k \)
\[
\bigcup_{ k \in \omega } ( R \restriction Y_k ) \cup \bigcup_{ k \in \omega } \setof{ ( y , y' ) \in Y_k \times Y_{ k + 1 }} { y \mathrel{R} x_{ n_k + 1 } \text{ and } x_{ n_{ k + 1 } } \mathrel{R} y' }.
\]
It readily follows from the definition that \( R_Y \subseteq R^+ \).

\begin{claim}
\( R_Y \) is total on \( \bigcup_{k \in \omega} Y_k \).
\end{claim}

\begin{proof}
Pick \( k \in \omega \) and \( y \in Y_k \), towards proving that \( y \in \dom ( R_Y ) \). 
Then there is a finite sequence \( \seq{ y_0 , \dots , y_m } \) of elements of \( Y_k \) such that \( x_{n_k} \mathrel{R} y_0 \mathrel{R} \cdots \mathrel{R} y_m \mathrel{R} x_{n_k + 1} \), and \( y = y_i \) for some \( 0 \leq i \leq m \).
If \( i < m \), then \( y \mathrel{R} y_{i + 1} \). 
If \( i = m \) then \( y \mathrel{R_Y} y' \) for any \( y' \in Y_{ k + 1} \) such that \( x_{n_{k+1}} \mathrel{R} y' \).
In either case \( y \in \dom ( R_Y ) \). 
\end{proof}

By \( \DC ( Y ) \), there is an \( R_Y \)-chain \( ( y_n )_{ n \in \omega} \). 
By part~\ref{prop:basicpropertiesofDC&AComega-b} of Proposition~\ref{prop:basicpropertiesofDC&AComega} we can suppose that \( y_0 \in Y_0 \) and that \( x_{n_0} \mathrel{R} y_0 \). 
As the \( Y_k \)s are disjoint, for every \( n \) there is a unique \( k \) such that \( y_n \in Y_k \), and let \( i ( n ) \) be this \( k \). 

\begin{claim}
The set \( I_k = \setof{ n \in \omega }{ i ( n ) = k } \) is a finite interval of natural numbers.
\end{claim}

\begin{proof}
By definition of \( R_Y \) it follows that either \( i ( n + 1 ) = i ( n ) \) or else \( i ( n + 1 ) = i ( n ) + 1 \), so it is enough to show that \( I_k \) is finite.
Towards a contradiction, suppose \( I_{\bar{k}} \) is infinite, for some \( \bar{k} \in \omega \).
This means that there is \( \bar{n} \) such that \( i ( n ) = i ( \bar{n} ) \) for all \( n \geq \bar{n} \), that is \( \setof{ y_n }{ n \geq \bar{n} } \subseteq Y_{\bar{k}} \).
But then \( R \restriction \setof{ y_n }{ n \geq \bar{n} } \) would be a total on \( \setof{ y_n }{ n \geq \bar{n} } \) and contained in \( R \restriction Y \), against~\eqref{eq:th:DCunion0}.
\end{proof}

Let \( m_k = \max ( I_k ) \) so that \( I_0 = [ 0 ; m_0 ] \) and \( I_{k + 1} = [ m_k + 1 ; m_{ k+1} ] \).
Then
\begin{multline*}
\seq{x_0 , \dots , x_{ n_0 } } \conc \seq{ y_0 , \dots , y_{ m_0 } } \conc \seq{ x_{ n_0 + 1 } , \dots , x_{ n_1 } }\conc \seq{ y_{ m_0 + 1 } , \dots , y_{ m_1 } } \conc \cdots
\\
\cdots \conc \seq{ x_{ n_k + 1 } , \dots , x_{ n_{ k + 1 } } }\conc \seq{ y_{ m_k + 1 } , \dots , y_{ m_{ k + 1 } } } \conc \cdots
\end{multline*}
is the required \( R \)-chain.
\end{proof}

\subsection{The Feferman-Levy model}\label{subsec:FefermanLevy}
Feferman and Levy showed that the following is consistent relative to \( \ZF \):
\begin{equation*}
 \R \text{ is the countable union of countable sets.} \tag{\( \FL \)}
\end{equation*}
(See~\cite{Jech:1973gf}*{p.\ 142} for an exposition of the Feferman-Levy model.)

The next result shows that in the Feferman-Levy model, the statement of Theorem~\ref{th:Lorenzo} fails, that is, there is no set \( A \subseteq \R \) such that \( \DC ( A ) \) and \( \neg \AComega ( A ) \).

\begin{proposition}\label{prop:FefermanLevy}
\( \FL \) implies that if \( \DC ( A ) \) holds with \( A \subseteq \R \), then \( A \) is countable.
\end{proposition}

We need a preliminary result.

\begin{lemma}\label{lem:FefermanLevy}
Assume \( \FL \).
Then there is a sequence of nonempty, countable, pairwise disjoint sets \( ( X_n )_{ n \in \omega } \) such that \( \R = \bigcup_{n} X_n \), and no infinite subsequence of \( ( X_n)_{ n \in \omega } \) has a choice function.
\end{lemma}

\begin{proof}
Fix a bijection \( \pi \colon \R \to \R^\omega \), and for each \( m \in \omega \) let \( \pi_m \colon \R \to \R \) be defined as \( \pi_m ( x ) = \pi ( x )_m \). 
If \( Y \subseteq \R \) and \( f \colon \omega \to Y \) is surjective, then \( \tilde{Y} \), the closure of \( Y \) under the \( \pi_m \)s, is also countable, as 
\[
 \tilde{f} \colon {}^{< \omega } \omega \times \omega \to \tilde{Y} \qquad ( \seq{ n_0 , \dots , n_k } , m ) \mapsto \pi_{n_k} \circ \cdots \circ \pi_{n_0}\circ f ( m ) 
\]
is surjective.
By \( \FL \) let \( ( Y_n )_{ n \in \omega} \) be a sequence of countable sets such that \( \R = \bigcup_n Y_n \), and without loss of generality we may assume that each \( Y_n \) is closed under every \( \pi_m \).
Then let \( X_n = Y_n \setminus \bigcup_{m < n} Y_m \) for each \( n \in \omega \). 
If necessary, we can pass to a subsequence to get them to be nonempty. 
 
We claim that no infinite subsequence of \( ( X_n )_{ n \in \omega } \) has a choice function. 
Otherwise there would be an infinite sequence \( ( x_n )_{ n \in \omega } \in \R^\omega \) whose range intersects infinitely many \( X_n \)s. 
Let \( x \in \R \) be such that \( \pi ( x) = ( x_n )_{n \in \omega } \).
Then \( x \in X_k \subseteq Y_k \) for some \( k \in \omega \), and hence 
\[
\FORALL{n \in \omega } \left ( x_n = \pi_n ( x ) \in Y_k \subseteq X_0 \cup \dots \cup X_k \right )
\] 
as \( Y_k \) is closed under the \( \pi_n \)s.
But this contradicts the assumption that \( \setof{ x_n }{ n \in \omega } \) intersects infinitely many \( X_n \)s.
\end{proof}

%By Lemma~\ref{lem:FefermanLevy} \( \FL \) implies the failure of \( \AComega ( \R ) \) in a strong sense.
%Let \( \wAComega ( X ) \) be the assertion that for any sequence \( ( A_n )_{ n \in \omega } \) of subsets of \( X \) there is an infinite \( I \subseteq \omega \) and \( ( a_i )_{i \in I} \) such that \( a_i \in A_i \) for all \( i \in I \).
%If moreover we assume the \( A_n \)s to be countable we obtain the weakening \( \wAComegaomega ( X ) \).
%Thus Lemma~\ref{lem:FefermanLevy} amounts to say that \( \FL \) implies the negation of \( \wAComegaomega ( \R ) \).

\begin{proof}[Proof of Proposition~\ref{prop:FefermanLevy}]
Fix \( ( X_n )_{ n \in \omega } \) as in Lemma~\ref{lem:FefermanLevy}. 
Let \( A \subseteq \R \) such that \( \DC ( A ) \) holds, and let \( I = \setof{ n \in \omega }{ A \cap X_n \neq \emptyset} \).
If \( I \) is infinite then, by part~\ref{prop:basicpropertiesofDC&AComega-d} of Proposition~\ref{prop:basicpropertiesofDC&AComega}, \( \DC ( A ) \) would imply the existence of a choice function for the family \( \setof{ A \cap X_n}{ n \in I } \), which is, in particular, a choice function for \( \setof{ X_n }{ n \in I } \), against Lemma~\ref{lem:FefermanLevy}. 
So \( I \) must be finite, that is \( A \subseteq X_0 \cup \dots \cup X_k \) for some \( k \).
But the finite union of countable sets is countable, so \( A \) is countable.
\end{proof}

\subsection{Definability of the counterexample}
Theorem~\ref{th:Lorenzo} shows that the statement~\eqref{eq:mainquestion2} is consistent with \( \ZF \), that is to say, it is consistent that there is a set \( A \subseteq \R \) such that \( \DC ( A ) \) and \( \neg \AComega ( A ) \).
The set \( A \) constructed in the proof of Theorem~\ref{th:Lorenzo} is a set of Cohen reals, so it is not ordinal definable.
But what is the possible descriptive complexity of a set \( A \) as above?

By part~\ref{lem:propertiesofA-c} of Proposition~\ref{lem:propertiesofA}, the set \( A \) cannot contain a perfect set.
Recall that a set has the perfect set property if it is either countable or else it contains a perfect subset.
Assuming \( \AComega ( \R ) \), every Borel set has the perfect-set property.
In a choice-less context the situation becomes murky.
Assuming \( \FL \), every set of reals is \( \Fsigmasigma \) (i.e.~countable union of \( \Fsigma \) sets), and by taking complements it is also \( \Gdeltadelta \) (i.e.~countable intersection of \( \Gdelta \) sets), so every set is \( \bDelta^{0}_{4} \), as \( \Fsigma = \bSigma^{0}_{2} \subset \bPi^{0}_{3} \), and hence \( \Fsigmasigma \subseteq \bSigma^{0}_{4} \).
Therefore \( \FL \) collapses the Borel hierarchy at level \( 4 \).
Moreover \( \FL \) implies that there is an uncountable set in \( \bPi^{0}_{3} \) without a perfect subset~\cite{Miller:2009xy}*{Theorem 1.3}.

On the other hand A.~Miller has shown in \( \ZF \) that \( \bSigma^{0}_{3} \neq \bPi^{0}_{3} \)~\cite{Miller:2008ys}*{Theorem 2.1}, and that every set in \( \bSigma^{0}_{3} \) has the perfect-set property \cite{Miller:2009xy}*{Theorem 1.2}.

Recall that a subset of \( \R \) is \( \bPi^{1}_{n} \) if it is the complement of a \( \bSigma^{1}_{n} \), and it is 
\( \bSigma^{1}_{n} \) if it is the projection of a \( \bPi^{1}_{n - 1} \) set \( C \subseteq \R \times \R \), where \( \bPi^{1}_{0} \) is the collection of closed sets.
The lightface hierarchy \( \varSigma^1_n , \varPi^1_n \) is obtained by replacing \( \bPi^{1}_{0} \) with \( \varPi^1_0 \), the collection of recursively-closed sets, see~\cite{Kanamori:2003zk}*{Ch.~3, \S 12}.
Working in \( \ZF \), every \( \bSigma^{1}_{1} \) set has the perfect set property, and by a theorem of Mansfield and Solovay (see~\cite{Kanamori:2003zk}*{Ch.~3, Corollary 14.9}) every \( \bSigma^{1}_{2} \) set is either well-orderable, being included in \( \Ll [ a ] \) for some real \( a \), or else it contains a perfect set.

By part~\ref{lem:propertiesofA-c} of Lemma~\ref{lem:propertiesofA} we obtain:

\begin{corollary}
If \( A \subseteq \R \) is \( \bSigma^{0}_{3} \) or \( \bSigma^{1}_{2} \) and \( \DC ( A ) \) holds, then \( \AComega ( A ) \).
\end{corollary}

We conclude with a question.

\begin{question}
Is it consistent with \( \ZF \) that there is a \( \varPi^{1}_{2} \) set \( A \subseteq \R \) such that \( \DC ( A ) \) and \( \neg \AComega ( A ) \)?
\end{question}

\end{document}